\documentclass[11pt,a4]{amsart}
\pdfoutput=1

\usepackage[textwidth=15.3cm, centering]{geometry}

\usepackage{amssymb}
\usepackage{amsmath}
\usepackage{mathtools}
\usepackage{microtype}
\usepackage{tikz}
\usetikzlibrary{arrows}
\usetikzlibrary{arrows.meta}
\usetikzlibrary{decorations.pathreplacing,decorations.markings}
\usetikzlibrary{positioning, 
                quotes}
\usepackage{booktabs}
\usepackage[pdftex,colorlinks,citecolor=black,linkcolor=black,urlcolor=black,bookmarks=false]{hyperref}
\usepackage{subfigure}
\usepackage{multirow,bigdelim}
\usepackage{array}

\usepackage{multicol}
\usepackage{bbm}

\newtheorem{theorem}{Theorem}
\newtheorem{proposition}[theorem]{Proposition}
\newtheorem{lemma}[theorem]{Lemma}

\theoremstyle{definition}

\numberwithin{theorem}{section}

\newcommand{\Z}{\mathbb{Z}}

\newcommand{\F}{\mathbb{F}}

\newcommand{\SL}{\operatorname{SL}}
\newcommand{\ord}{\operatorname{ord}}
\newcommand{\pa}{\alpha}

\DeclareMathOperator{\tr}{tr}

\title[Markoff divisibility]{Divisibility by $p$ for Markoff-like Surfaces}

\author[M. de Courcy-Ireland]{Matthew de Courcy-Ireland}
\email{matthew.decourcy-ireland@math.su.se}
\address{Department of Mathematics, Stockholm University and the Nordic Institute for Theoretical Physics}
\author[M. Litman]{Matthew Litman}
\email{matthew.litman@ucd.ie}
\address{School of Mathematics and Statistics, University College Dublin}
\author[Y. Mizuno]{Yuma Mizuno}
\email{mizuno.y.aj@gmail.com}
\address{School of Mathematical Sciences, University College Cork}
\date{\today}%

\begin{document}

\begin{abstract}
    We study orbits in a family of Markoff-like surfaces with extra off-diagonal terms over prime fields $\mathbb{F}_p$.
    It is shown that, for a typical surface of this form, every non-trivial orbit has size divisible by $p$.
    This extends a theorem of W.Y. Chen from the Markoff surface itself to others in this family.
    The proof closely follows and elaborates on a recent argument of D.E. Martin.
    We expect that there is just one orbit generically.
    For some special parameters, we prove that there are at least two or four orbits.
    Cayley's cubic surface plays a role in parametrising the exceptional cases and dictating the number of solutions mod $p$.
\end{abstract}

\maketitle

\section{Introduction} \label{sec:intro}

A family of surfaces defined by the equation
\begin{equation} \label{eqn:family}
x_1^2 + x_2^2 + x_3^2 + a_1x_2x_3+a_2x_1x_3+a_3x_1x_2 = (3+a_1+a_2+a_3)x_1x_2x_3
\end{equation}
has been studied by Gyoda and Matsushita \cite{GM} for integers $a_i \geq 0$, with a special case $a=(1,1,1)$ already appearing in \cite{G} and the classical case $a=(0,0,0)$ going back to the well-known Markoff tree \cite{Markoff}.
By construction, $x=(1,1,1)$ is a solution for any choice of parameters $a_1$, $a_2$, and $a_3$.
The following moves construct new solutions from old, by changing a single variable to the other root of the quadratic equation (\ref{eqn:family}).
The moves can be written as
\begin{equation} \label{eqn:move}
m_i: x_i \mapsto -x_i + sx_{i-1}x_{i+1} - a_{i+1}x_{i-1}-a_{i-1}x_{i+1}
\end{equation}
where the index $i$ is interpreted modulo 3, and
\begin{equation} \label{eqn:sum}
s=3+a_1+a_2+a_3.
\end{equation}
Gyoda and Matsushita showed in \cite[Theorem 1.1]{GM} that all solutions in positive integers are given by repeatedly applying these moves to $x=(1,1,1)$.
Our interest is in how the integer solutions reduce modulo $p$.
Is every solution over $\F_p$ the reduction mod $p$ of a solution over $\Z$?

Recently there has been progress answering this question for the Markoff surface, which is a special case of (\ref{eqn:family}) where $a_1=a_2=a_3=0$. 
Chen \cite{Chen} showed that except for $(x_1,x_2,x_3)=(0,0,0)$, all orbits under the three moves (\ref{eqn:move}) have size divisible by $p$.
In particular, a non-trivial orbit must have size at least $p$. Combined with the work of Bourgain--Gamburd--Sarnak \cite{BGS}, this lower bound implies that for any sufficiently large prime $p$ there is just one orbit besides $(0,0,0)$ containing all of the $p^2 \pm 3p$ other solutions.
It is enough for $p$ to have a few hundred digits, more precisely $p > 3.45 \cdot 10^{392}$ is sufficient as shown in \cite{EFLMT}.

Martin \cite{Martin} gave an elementary proof of Chen's congruence for the Markoff surface. Our main result shows how it can be adapted to the situation (\ref{eqn:family}).
From now on we fix a prime $p$ and consider $a_i \in \F_p$ as elements of a finite field rather than integers.

\begin{theorem} \label{thm:congruence-lower-bound}
Assume $3+a_1+a_2+a_3 \neq 0$ and, for all $i=1,2,3$, $a_i^2 \neq 4$ in $\F_p$ for a prime $p \geq 5$.
Then, except for the orbit of size $1$ containing $(0,0,0)$, any orbit under the three moves (\ref{eqn:move}) has size divisible by $p$.

If $a_i^2=4$ for some $i$ and
\begin{equation} \label{eqn:hypo}
2a_{i-1} = a_{i+1} a_i
\end{equation}
then any orbit has size divisible by $p$.
\end{theorem}

The hypothesis (\ref{eqn:hypo}) means the parameters are, up to permutation, of the form $(2\sigma, \pa, \pa \sigma)$ for some $\pa \in \F_p$ and some $\sigma=\pm 1$.
This includes four special values
\begin{equation} \label{eqn:cayley-nodes}
(a_1,a_2,a_3)=(2,2,2), \quad (2,-2,-2), \quad (-2,2,-2), \quad (-2,-2,2)
\end{equation}
together with the generic case where $\pa^2 \neq 4$.
For the special values, we suspect that there are four orbits of sizes depending on $p \bmod 4$.
Their sizes can be written using the quadratic character $\chi$ modulo $p$ as follows:
\[
p^2=p \frac{p+3\chi(-1)}{4} +  p \frac{p-\chi(-1)}{4}+p \frac{p-\chi(-1)}{4}+p \frac{p-\chi(-1)}{4}
\]
where $p^2$ is the total number of non-zero solutions $x$, partitioned into one orbit of size $p(p\pm 3)/4$ and three orbits of size $p(p \mp 1)/4$.
For $(2\sigma,\pa, \pa \sigma)$ with $\pa^2 \neq 4$, there seem to be only two orbits:
\[
p^2+\chi(\pa^2-4)p=p\frac{p-\chi(\pa^2-4)}{2} + p \frac{p+3\chi(\pa^2-4)}{2}
\]
If $a_i^2 \neq 4$, for all $i$, then we believe there is only one orbit. Its size is
\begin{equation} \label{eqn:numel}
    p^2+ p\big(\sum_{i=1}^3 \chi(a_i^2-4)+C(a_1,a_2,a_3) \big)
\end{equation}
where $C: \F_p^3 \rightarrow \{0,1,-1\}$ is an extra $\pm 1$ present when $(a_1,a_2,a_3)$ lies on Cayley's cubic surface. Proposition~\ref{prop:numel} shows that, regardless of the orbit structure, (\ref{eqn:numel}) is the total number of solutions $x \neq (0,0,0)$ to (\ref{eqn:family}).

We can show that there are at least this many orbits because of a quadratic obstruction derived from equation (\ref{eqn:family}).
A similar obstruction for a related family of Markoff-type K3 surfaces was described by O'Dorney in terms of a double-cover of the surface~\cite{O}.

\begin{theorem} \label{thm:breakup}
    If $a_i=2\sigma, a_{i-1}=\pa,$ and $a_{i+1}=\pa\sigma$ where $\sigma=\pm1$, then (\ref{eqn:family}) has at least two orbits besides $(0,0,0)$.
    If $\pa=\pm 2$, i.e., the cases (\ref{eqn:cayley-nodes}), then (\ref{eqn:family}) has at least four orbits besides $(0,0,0)$.
\end{theorem}
Not only do the methods of Section~\ref{section:quadobstruction} show that the graph splits into 2 or 4 pieces, they demonstrate that the pieces are of roughly the same size.
One can hope to prove that these do not split any further by adapting the methods of Bourgain--Gamburd--Sarnak \cite{BGS}. At present, we have verified by computer for every prime $5 \leq p < 200$ that the graph has one component in the generic case ($a_i^2\neq 4$ for all $i$) and that the special cases in Theorem~\ref{thm:breakup} do not split any further. We intend to pursue this in a follow-up paper.

If (\ref{eqn:hypo}) fails, then the orbit sizes are not necessarily divisible by $p$. For example, the singleton $(0,a-b,b-a)$ lies in its own orbit in the case of parameters $(2,a,b)$.
Likewise if $3+a_1+a_2+a_3 = 0$, then the orbits are not necessarily divisible by $p$. This already happens for the Markoff surface itself and $p=3$, where there is a single orbit of size 8 apart from $(0,0,0)$.
This example can be extended to all primes by taking $a_1=a_2=0$ and $a_3=-3$ for any $p$, which we discuss in Section~\ref{sec:counter}.
Proposition~\ref{prop:00-3} gives the number of orbits.

One could consider more generally
\[
x_1^2 + x_2^2 + x_3^2 + a_1x_2x_3+a_2x_1x_3+a_3x_1x_2 = sx_1x_2x_3
\]
where $s$ is not necessarily given by the sum (\ref{eqn:sum}).
However, over a field, this degree of freedom can be eliminated up to a scaling.
Changing $x \mapsto tx$ with $t \neq 0$ and then cancelling $t^2$ from both sides preserves $a_1$, $a_2$, $a_3$ while converting $s$ to $ts$.
Given any non-zero value of $3+a_1+a_2+a_3 \neq 0$, we may rescale to the situation (\ref{eqn:sum}).
The case $3+a_1+a_2+a_3=0$ should be considered separately.

\subsection{Relation to generalised cluster algebras}

The Markoff surface, where $a_1=a_2=a_3=0$ in (\ref{eqn:family}), has a rich connection to cluster algebras with achievements including proofs of Aigner's monotonicity conjectures \cite{LLRS, RS}. 

Non-zero parameters lead to a ``generalised cluster algebra'' instead, for which we refer to \cite{BG, BV, GM}.
Figure~\ref{fig:quivers} shows the quivers associated with Markoff type surfaces.
For the right quiver, which yields generalised cluster algebras, there is a polynomial $1+a_i z + z^2$ at each node, called the exchange polynomial.
The moves of the equation \eqref{eqn:family} come from mutations in the theory of generalised cluster algebras, which are described by the relations given by the exchange polynomials as follows:
\begin{equation}\label{eqn:move multiplicative}
  \begin{split}
    x_i' x_i &= x_{i-1}^2 + a_i x_{i-1} x_{i+1} + x_{i+1}^2 \\
    \big(&= x_{i+1}^2 (1 + a_i (x_{i-1} x_{i+1}^{-1}) + (x_{i-1} x_{i+1}^{-1})^2) \big)
  \end{split}
\end{equation}
where $x_i'$ is the new variable obtained by the mutation at $x_i$.
This relation can be interpreted as a quadratic equation jointly in $x_i'$ and $x_i$ with coefficients depending on $x_{i-1}$ and $x_{i+1}$.

The exchange polynomial plays a role in the proof of Theorem~\ref{thm:congruence-lower-bound}.
Triples with $x_i=0$ form cycles of length given by the order of this polynomial's roots in $\F_p^{\times}$ or $\F_{p^2}^{\times}$.

\tikzset{
arrow/.style={postaction={decorate,decoration={
  markings,
  mark=at position .55 with {\arrow[#1]{To[length=1.1mm]}}
}}},
double arrow/.style={postaction={decorate,decoration={
  markings,
  mark=at position .50 with {\arrow[#1]{To[length=1.1mm]}},
  mark=at position .60 with {\arrow[#1]{To[length=1.1mm]}}
}}},}
\begin{figure}
\begin{tikzpicture}[scale=1.5]
\pgfmathsetmacro{\r}{0.75}
\draw (0,0)++(90:1) node(1){$f$};
\draw (0,0)++(210:1) node(2){$f$};
\draw (0,0)++(-30:1) node(3){$f$};
\draw [double arrow] (1)--(2);
\draw [double arrow] (2)--(3);
\draw [double arrow] (3)--(1);
\draw (0,0)++(-90:1.5*\r) node{$f(z)=1+z$};
\end{tikzpicture}
\hspace{1cm}
\begin{tikzpicture}[scale=1.5]
\pgfmathsetmacro{\r}{0.75}
\draw (0,0)++(90:1) node(1){$f_1$};
\draw (0,0)++(210:1) node(2){$f_2$};
\draw (0,0)++(-30:1) node(3){$f_3$};
\draw [arrow] (1)--(2);
\draw [arrow] (2)--(3);
\draw [arrow] (3)--(1);
\draw (0,0)++(-90:1.5*\r) node{$f_i(z)=1+a_iz+z^2$};
\end{tikzpicture}
\caption{Left: the quiver for the cluster algebra associated with the classical Markoff equation. 
Right: the quiver for the generalised cluster algebra associated with the generalised Markoff equation 
with parameters $a_1$, $a_2$, $a_3$.
The $f$ and $f_i$ are called exchange polynomials at the corresponding vertices.
}
\label{fig:quivers}
\end{figure}
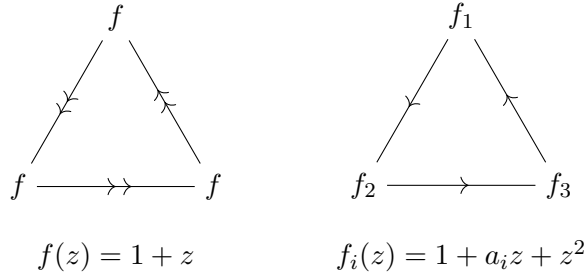

\subsection{Interpretation as a character variety}

Beyond rescaling, a certain affine transformation can be used to rewrite (\ref{eqn:family}).
Define new variables $u=(u_1,u_2,u_3)$ by
\begin{equation} \label{eqn:ui}
u_i = sx_i -a_i
\end{equation}
Then (\ref{eqn:family}) becomes
\begin{equation} \label{eqn:linear-terms}
\sum_i \big(u_i^2 + (2a_i + a_{i-1}a_{i+1}) u_i \big) = u_1u_2u_3 - 2a_1a_2a_3-a_1^2-a_2^2-a_3^2
\end{equation}
where the lower-order terms are quadratic in $a$ and linear in $u$, reversing the roles compared to $a_i x_j x_k$ from (\ref{eqn:family}).
The linear version (\ref{eqn:linear-terms}) has a known interpretation as the character variety of a four-holed sphere. 
See \cite{benny-goldman,Magnus,MPT} and the classical works of Fricke and Vogt \cite{Vogt}.
See also \cite{Cantat, Cantat-Loray} and the recent preprint \cite{CDMB}.
In this interpretation, one can think of
\[
u_1=-\tr(AB), \quad u_2=-\tr(BC), \quad u_3=-\tr(CA)
\]
where $A$, $B$, $C$ are matrices in $\SL_2(\F_p)$ such that
\[
a_1=\tr(A), \quad a_2=\tr(B), \quad a_3=\tr(C), \quad \tr(ABC)=2.
\]

One can change $u_i$ to
\begin{equation} \label{eqn:miu}
\widetilde{u}_i = -u_i + u_{i-1}u_{i+1} - 2a_i - a_{i-1}a_{i+1}
\end{equation}
which gives another solution to (\ref{eqn:linear-terms}), keeping $u_{i \pm 1}$ the same.
These moves commute with the change of variable, that is,
\begin{equation} \label{eqn:equivariant}
\widetilde{u}_i = sx_i' - a_i
\end{equation}
Indeed, substituting $u=sx-a$ into (\ref{eqn:move}), we find
\begin{align*}
sx_i'-a_i = -u_i-2a_i +(u_{i-1}+a_{i-1})(u_{i+1}+a_{i+1}) \\-a_{i-1}(u_{i+1}+a_{i+1})
-a_{i+1}(u_{i-1}+a_{i-1})
\end{align*}
which simplifies to (\ref{eqn:miu}).

The change of variables \eqref{eqn:miu} appears in \cite[Section 6]{CMR}, where they use it to study the relation between the character variety equation \eqref{eqn:linear-terms} and the generalised cluster algebra associated with the right quiver in Figure~\ref{fig:quivers}.
See also \cite{GMS, Hikami} for the relation between the character variety equation and the (generalised) cluster algebra.

\subsection{Graphs}

Orbits can be thought of as connected components in a graph where a vertex $x=(x_1,x_2,x_3)$ is adjacent to $m_i x$ for each of the moves (\ref{eqn:move}) for $i=1,2,3$.
The following proposition shows that there are no bigons in these graphs.
If there are two different edges between vertices $x$ and $y$, then in fact $x=y$ is a fixed point for both the corresponding moves.
This fact is important for the proof of Theorem~\ref{thm:congruence-lower-bound} in the case of parameters of the form $a=(2\sigma,\pa,\pa\sigma)$.
However, for the statement, one can take any triples $(x_1,x_2,x_3)$ in $\F_p^3$ as vertices, regardless of whether they lie on the surface (\ref{eqn:family}).
\begin{proposition}\label{prop:no-bigons}
If $m_ix = m_j y$ where $i \neq j$ and $x,y \in \F_p^3$, then $x=y$.
\end{proposition}

\begin{proof}
Suppose $i=1$ and $j=2$.
Since each move changes only one coordinate, from $m_1x=y$ it follows that $x$ and $y$ agree in the second and third coordinates.
Similarly from $m_2x=y$, they agree in the first and third coordinates.
Thus $x$ and $y$ agree in all coordinates.
\end{proof}

\section*{Acknowledgments}
The authors thank Robert Osburn, who suggested using cluster algebras to generalize \cite{BGS}, which served as the starting point for this project.
MdCI and ML are thankful to have participated in the workshop
\href{https://www.birs.ca/events/2025/5-day-workshops/25w5411}{\textit{Perspectives on Markov Numbers}}
at the Banff International Research Station,
where Martin's method was first presented.

ML and YM were partially supported by the Irish Research Council Advanced Laureate Award IRCLA/2023/1934.
MdCI was supported by the Knut and Alice Wallenberg Foundation (Wallenberg Initiative for Networks and Quantum Information) and by Stiftelsen Hierta Retzius through the Royal Swedish Academy of Sciences.

\section{Divisibility of orbit sizes}

In this section, we prove the main result Theorem~\ref{thm:congruence-lower-bound}. The argument is closely modelled on Martin's proof from \cite{Martin}.

For simplicity, let us first give a proof assuming that $x_i \neq 0$ for all $(x_1,x_2,x_3)$ in a given orbit. It is also helpful to imagine that $m_i x \neq x$ for all $x$ in the orbit.
As we will argue later, these assumptions can be removed as long as $a_i^2 \neq 4$ as well as in the cases where $a_i^2=4$ and (\ref{eqn:hypo}) holds.
However, the argument is easier to write with them in mind.

It is also worth noting that, if $a_i^2-4$ is not a square in $\F_p$, then $(0,0,0)$ is the only solution with $x_i=0$.
Thus the simple version of the proof assuming $x_i \neq 0$ already gives many cases of the result.

\begin{proof}
With $s=3+a_1+a_2+a_3$ and taking the indices cyclically, we can write (\ref{eqn:family}) as
\[
x_1^2+x_2^2+x_3^2+ \sum_{i \bmod 3} a_i x_{i-1}x_{i+1} = s x_1x_2x_3
\]
If no $x_i$ vanishes, then dividing both sides by $x_1x_2x_3$ gives
\begin{equation} \label{eqn:predelta}
\sum_{i \bmod 3} \big( \frac{x_i}{x_{i-1}x_{i+1}} + \frac{a_i}{x_i} \Big)= s.
\end{equation}
We have
\[
\sum_{i \bmod 3} \frac{a_i}{x_i} = \frac{1}{2} \sum_{j \bmod 3} \Big( \frac{a_{j-1}}{x_{j-1}} + \frac{a_{j+1}}{x_{j+1}} \Big)
\]
since the second sum $\sum_{j \bmod 3}$ amounts to summing $a_i/x_i$ with each index counted twice (once as $i-1+1$ and again as $i+1-1$).

This motivates the definition, for $x_1x_2x_3 \neq 0$, of three functions
\begin{equation} \label{eqn:delta}
\Delta_i(x) \coloneqq \frac{x_i}{x_{i-1}x_{i+1}} + \frac{1}{2} \Big( \frac{a_{i-1}}{x_{i-1}}+\frac{a_{i+1}}{x_{i+1}} \Big).
\end{equation}
From (\ref{eqn:family}), or rather (\ref{eqn:predelta}), they satisfy
\begin{equation} \label{eqn:delta-total}
\sum_{i \bmod 3} \Delta_i(x) = s.
\end{equation}
From (\ref{eqn:move}), it follows that
\begin{equation} \label{eqn:delta-pair}
\Delta_i(x) + \Delta_i(m_ix) = s
\end{equation}
for each index $i$.
If $x=m_i x$ is fixed by a move, then for that index
\begin{equation} \label{eqn:delta-fix}
\Delta_i(x) = \frac{s}{2}.
\end{equation}

If a particular coordinate vanishes, say $x_i = 0$, then (\ref{eqn:delta}) can be taken as a definition of the corresponding $\Delta_i(x)$.
The same formula for the other two $\Delta_{i \pm 1}(x)$ would involve division by $0$. However, as we argue in the next sections, the definition can be extended so that (\ref{eqn:delta-total}) and (\ref{eqn:delta-pair}) continue to hold.
See 
\eqref{even} and \eqref{odd} for a suitable extension.

Assuming this extension is possible, here is how to prove Theorem~\ref{thm:congruence-lower-bound}.
Let $V$ be the number of points $(x_1,x_2,x_3)$ in an orbit $O$.
By (\ref{eqn:delta-total}),
\[
sV = \sum_{x \in O} s = \sum_{x \in O} \sum_{i \bmod 3} \Delta_i(x) = \sum_{i} \sum_{x \in O} \Delta_i(x).
\]
Since $O$ is an orbit, we have $m_i x \in O$ for each $x \in O$.
It is a union of pairs $\{x,m_i x\}$ together, perhaps, with some singletons $x=m_i x$.
By (\ref{eqn:delta-pair}) and (\ref{eqn:delta-fix}),
\[
\sum_{x \in O} \Delta_i(x) = \sum_{\{x,m_i x\}} s + \sum_{x = m_i x} \frac{s}{2} = \frac{sV}{2}.
\]
Combining these steps, we get the same sum three times
\[
sV = \sum_{x \in O} s = \sum_{x \in O} \sum_{i \bmod 3} \Delta_i(x) = \sum_{i} \frac{sV}{2} = \frac{3sV}{2}
\]
and so $sV=0$. This implies $V = 0$ as we have assumed $s \neq 0$.
\end{proof}

\subsection{Vanishing coordinates}

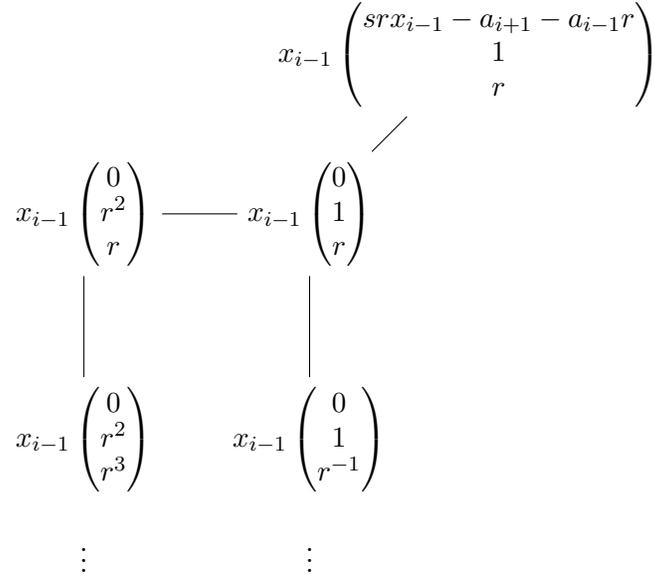
\begin{figure}
\begin{tikzpicture}[scale=3]
\draw (0,0) node(0){$x_{i-1}\begin{pmatrix} 0 \\ 1 \\ r \end{pmatrix}$};
\draw (0,-1) node(1){$x_{i-1}\begin{pmatrix} 0 \\ 1 \\ r^{-1}\end{pmatrix}$};
\draw (-1,0) node(2){$x_{i-1}\begin{pmatrix} 0 \\ r^2 \\ r\end{pmatrix}$};
\draw (-1,-1) node(3){$x_{i-1}\begin{pmatrix} 0 \\ r^2 \\ r^3\end{pmatrix}$};
\draw (0)++(45:1) node(4){$x_{i-1}\begin{pmatrix} srx_{i-1}-a_{i+1}-a_{i-1}r \\ 1 \\ r \end{pmatrix}$};
\draw (1)--(0)--(2)--(3);
\draw (0)--(4);
\draw (3)++(270:0.5) node{$\vdots$};
\draw (1)++(270:0.5) node{$\vdots$};
\end{tikzpicture}
\caption{
The action of $m_{i-1}$ and $m_{i+1}$ on triples with $x_i=0$.
For each value of $x_{i-1}$, they form a cycle whose length depends on the order of a solution to $r^2+a_i r + 1 = 0$.
If $a_i=0$, then the cycle has length 4.
}
\label{fig:xi0}
\end{figure}

In order to define $\Delta_{i\pm 1}(x)$ for triples with $x_i=0$, it is useful to describe these triples in more detail.
Throughout this section, we let $x$ be a point with $x_i=0$ for some $i$.
First note that if two coordinates vanish, then the third must also vanish by (\ref{eqn:family}).
Excluding $(0,0,0)$, only a single coordinate can vanish.
If $x_i= 0$, then (\ref{eqn:family}) simplifies to
\begin{equation} \label{eqn:xi0}
x_{i-1}^2 + x_{i+1}^2 + a_ix_{i-1}x_{i+1} = 0.
\end{equation}
Assuming $a_i^2-4 \neq 0$, there are two roots
\begin{equation} \label{eqn:ri}
r_i = \frac{-a_i + \sqrt{a_i^2-4}}{2}
\end{equation}
either of which can be chosen when solving (\ref{eqn:xi0}) for
\[
x_{i+1} = r_i x_{i-1}.
\]
The quadratic satisfied by $r_i$ is
\begin{equation} \label{eqn:ri-implicit}
r_i^2 + a_i r_i + 1 = 0
\end{equation}
which may be interesting to note in connection with the generalised cluster algebra underlying (\ref{eqn:family}).
It follows that $r_i \neq 0$ and the other solution of the quadratic is $r_i^{-1}$.

The conic $x_i=0$ therefore degenerates to a pair of lines meeting at $(0,0,0)$. Excluding the origin, there are $2(p-1)$ non-zero solutions.
They can be grouped into cycles under the action of $m_{i-1}$ and $m_{i+1}$.
These cycles are related to each other by scaling: since (\ref{eqn:xi0}) is homogeneous, we may for instance assume $x_{i-1}=1$ or use $x_{i-1}$ to parametrise the lines.
Each move exchanges the two lines, that is, changes $r_i$ to $r_i^{-1}$.
A double move $m_{i-1} \circ m_{i+1}$ scales the parametrisation of each line by $r_i^2$.

The moves $m_{i \pm 1}$ generate a dihedral group acting on the conic $\{x_i=0\}$, with cyclic subgroup generated by
\begin{equation} \label{eqn:rhotation}
\rho = m_{i+1}m_{i-1}.
\end{equation}
Indeed one can check the dihedral relation
\[
m_{i-1}\rho = m_{i-1}m_{i+1}m_{i-1} = \rho^{-1}m_{i-1}.
\]
We have singled out $m_{i-1}$ in this description. The other move is given by
\[
m_{i+1} = m_{i-1} \rho^{-1} = m_{i-1} \rho^{N-1}
\]
if $\rho$ has order $N$. Indeed, since the moves are involutions,
\[
m_{i-1} \rho^{-1} = m_{i-1} m_{i-1} m_{i+1} = m_{i+1}.
\]

The order $N$ is some divisor of $p-\chi(a_i^2-4)$, that is, $p-1$ if $r_i \in \F_p$ or $p+1$ if $r_i$ lies in a quadratic extension. We assume $\chi(a_i^2-4)=1$ or else $(0,0,0)$ is the only triple with $x_i=0$.
The pair of lines where $x_i=0$ consists of $2(p-1)$ points divided into $\frac{p-1}{N}$ cycles.

The cycle is then of length $2N$, but to see the dihedral symmetry it may be better to visualise an $N$-sided polygon with $\rho^k x$ at the corners for $k=0,\ldots,N-1$. 
The other points $m_{i-1}\rho^k x$ can be thought of as midpoints of the edges of the polygon, or as a second $N$-gon.

However, the case $N=1$ is special in various ways. The polygon degenerates to a single vertex with two self-edges. We record a lemma stating how this occurs.
\begin{lemma}\label{lemma:N=1 iff a_i^2=4}
  Assume that $x \neq (0,0,0)$ and $x_i = 0$.
  The following are equivalent:
  \begin{alignat*}{3}
      &(1) \ a_i^2 = 4, \qquad
      &&(4) \ r_i = -\frac{a_i}{2},\qquad
      &&(7) \ x_{i-1}^2 = x_{i+1}^2, \\
      &(2) \ r_i^2 = 1, \quad
      &&(5) \ r_i^{-1} = -\frac{a_i}{2}, \quad 
      &&(8) \ m_{i-1} x = x, \\
      &(3) \ N = 1, \quad
      &&(6) \ x_{i-1} + \frac{a_i}{2} x_{i+1} = 0, \qquad 
      &&(9) \ m_{i+1} x = x. \\
  \end{alignat*}
\end{lemma}
\begin{proof}
    The equivalence of (1), (2), (4), and (5) follows from the equation \eqref{eqn:ri}.
    Since $N$ is the order of $r_i^2$, we also have (2) if and only if (3).
    By completing the square for the equation \eqref{eqn:xi0}, we have (1) if and only if (6).
    By the equation \eqref{eqn:xi0}, we also have (6) if and only if (7).
    Note that $m_{i-1}$ acts by
    \[
    x_{i-1} \mapsto -x_{i-1}-a_ix_{i+1}
    \]
    so the fixed point equation in (8) implies (6).
    By Proposition~\ref{prop:no-bigons}, we have (3) implies (9).
    Thus, we see that (8) implies (9).
    The same argument for $m_{i+1}$ shows that (9) implies (8).
\end{proof}

\begin{proposition}\label{prop:cycle nondegenerate}
  Assume that $x \neq (0,0,0)$, $x_i = 0$, and $N \geq 2$.
  The $2N$ cycle is non-degenerate, that is, the $2N$ points 
  \begin{equation}\label{eq:cycle points}
    \begin{split}
    &x, \rho x, \ldots, \rho^{N-1} x, \\
    &m_{i-1} x, m_{i-1} \rho x, \ldots, m_{i-1} \rho^{N-1} x
    \end{split}
  \end{equation}
  are distinct.
\end{proposition}
\begin{proof}
  We first note that we have $a_i^2 \neq 4$ by Lemma~\ref{lemma:N=1 iff a_i^2=4}.
  Suppose contrarily that $x$ is equal to another point $y$ in the cycle.
  Since $N \geq 2$, the $y$ cannot be $\rho^{l} x$ for some $1 \leq l \leq N-1$.
  Thus we can assume that $y = m_{i-1} \rho^l x$ for some $0 \leq l \leq N-1$.
  We consider two cases where $x = m_{i-1} \rho^{2 l} x$ or $x = m_{i-1} \rho^{2 l + 1} x$.
  In the first case, we have $\rho^l x = m_{i-1} \rho^l x$ by applying $\rho^l$ to $x = m_{i-1} \rho^{2 l} x=\rho^{-\ell}m_{i-1}\rho^l x$, using $m_{i-1}\rho=\rho^{-1}m_{i-1}$. This implies $a_i^2 = 4$ by Lemma~\ref{lemma:N=1 iff a_i^2=4}.
  In the second case, we similarly have $m_{i-1}\rho^{2\ell+1}x=m_{i+1}\rho^{2\ell+2}x$ and $m_{i+1} \rho^{l+1} x = \rho^{l+1} x$, which also implies $a_i^2 = 4$ by Lemma~\ref{lemma:N=1 iff a_i^2=4}.
  In both cases, we arrive at a contradiction.
\end{proof}

\subsection{Completing the proof}

It remains to show that $\Delta_{i\pm 1}(x)$ can be defined even when $x_i = 0$ so that (\ref{eqn:delta-total}) and (\ref{eqn:delta-pair}) hold.
In fact, there is a degree of freedom in doing so.
If $N \geq 2$, then any starting value $\Delta_{i-1}(x)$ can be propagated around the cycle by imposing (\ref{eqn:delta-total}) and (\ref{eqn:delta-pair}).
If $N=1$ there is no choice, and for some parameters, no solution at all.

Consider a triple $x$ where $x_i=0$. We will extend $\Delta_{i-1}$ to the orbit of $x$ under $m_{i+1}$ and $m_{i-1}$, then define $\Delta_{i+1}$ so that (\ref{eqn:delta-total}) holds.
To start, choose any value $\delta \in \F_p$ and define $\Delta_{i-1}(x) = \delta$.
For $n \geq 0$, define
\begin{align}
\Delta_{i-1}(\rho^n x) &= \Delta_{i-1}(x) - \sum_{\ell=0}^{n-1} \Big(\Delta_i (m_{i-1}\rho^{\ell}x) + \Delta_i (\rho^{\ell+1} x) \Big) \label{even} \\
\Delta_{i-1}(m_{i-1}\rho^n x) &= s-\Delta_{i-1}(x) + \sum_{\ell=0}^{n-1} \Big(\Delta_i (m_{i-1}\rho^{\ell}x) + \Delta_i (\rho^{\ell+1} x) \Big) \label{odd}
\end{align}
where, for $n=0$, the empty sums on the right are interpreted as $0$.
This procedure might seem to give two values for $\Delta_{i-1}$ at $x = \rho^N x$ where $N$ is the order of $\rho$. However, they agree because of the following fact.

\begin{proposition} \label{prop:delta-cycle}
Let $m_{i \pm 1}$ be the moves (\ref{eqn:move}) acting on solutions $x \neq (0,0,0)$ to (\ref{eqn:family}) with $x_i = 0$, and suppose $N$ is the order of $\rho = m_{i+1} m_{i-1}$.
Then the $\Delta_i$ from (\ref{eqn:delta}) satisfy
\begin{equation}\label{eq:delta-cycle}
\sum_{\ell=0}^{N-1} \Big(\Delta_i (m_{i-1}\rho^{\ell}x) + \Delta_i (\rho^{\ell+1} x) \Big) = 0
\end{equation}
when $a_i^2 \neq 4$.
Moreover, when $a_i^2 = 4$, the relation \eqref{eq:delta-cycle} holds if and only if
$2a_{i-1} = a_{i+1} a_i$.
\end{proposition}
\begin{proof}
First we simplify (\ref{eqn:delta}) to
\[
x_i = 0 \implies \Delta_i(x) = \frac{1}{2} \Big( \frac{a_{i-1}}{x_{i-1}}+\frac{a_{i+1}}{x_{i+1}} \Big).
\]
In the same way all around the cycle, $\Delta_i (gx)$ is defined for all $g$ in the dihedral group $\langle m_{i-1}, m_{i+1} \rangle$.

The dihedral action has a simple effect on $\Delta_i(x)$.
If $x_i=0$, then $x_{i+1}/x_{i-1}=r$ where $r^2+a_i r + 1 = 0$.
For $x$ on the line corresponding to a specific choice of $r$, we have
\[
\Delta_i(x) = \frac{1}{2}\Big( \frac{a_{i-1}}{x_{i-1}}+\frac{a_{i+1}}{x_{i+1}}\Big) = \frac{1}{2x_{i-1}} (a_{i-1}+a_{i+1}r^{-1}).
\]
The rotation $\rho = m_{i+1}m_{i-1}$ changes $x_{i-1}$ to $r^2 x_{i-1}$ and therefore
\[
\Delta_i(\rho x) = r^{-2} \Delta_i(x).
\]

If $a_i^2 \neq 4$, we have $r^{-2} \neq 1$ and thus $1+r^{-2}+\ldots + (r^{-2})^{N-1}=0$. We now have
\begin{align*}
&\sum_{\ell=0}^{N-1} \Big(\Delta_i (m_{i-1}\rho^{\ell}x) + \Delta_i (\rho^{\ell+1} x) \Big) = \\
&\big(1 + r^{-2} + \ldots + (r^{-2})^{N-1}\big)\big( \Delta_{i-1}(m_{i-1}x) + \Delta_i(x) \big) = 0.
\end{align*}

If $a_i^2 = 4$, we have $N=1$ by Lemma~\ref{lemma:N=1 iff a_i^2=4}.
Thus it suffices to show 
\begin{equation}\label{eq:delta-cycle for a_i^2=4}
    \Delta_i(m_{i-1} x) + \Delta_i (x) = 0.
\end{equation}
By Lemma~\ref{lemma:N=1 iff a_i^2=4}, we have $m_{i-1} x = x$ and $r^{-1} = - a_i / 2$,
and the left-hand side of \eqref{eq:delta-cycle for a_i^2=4} is computed as
\begin{align*}
    \Delta_i(m_{i-1} x) + \Delta_i (x) = 2 \Delta_i (x)
    &= \frac{1}{x_{i-1}} (a_{i-1} + a_{i+1} r^{-1})\\
    &= \frac{1}{2 x_{i-1}} (2 a_{i-1} - a_{i+1} a_i).
\end{align*}
This vanishes if and only if $2 a_{i-1} = a_{i+1} a_i$.
\end{proof}

In the special case $N=1$, one has not only $x=\rho x$ but in fact $x=m_{i-1}x$ and $x=m_{i+1}x$ from Lemma~\ref{lemma:N=1 iff a_i^2=4}.
Therefore (\ref{even}) and (\ref{odd}) force us to take
\begin{equation*}
    \Delta_{i-1} (x) = \Delta_{i-1} (m_{i-1} x) = \frac{s}{2},
\end{equation*}
and then $\Delta_i(x)=0$.
As long as $2a_{i-1} = a_{i+1}a_{i}$, the resulting functions $\Delta_i$ solve (\ref{eqn:delta-total}) and (\ref{eqn:delta-pair}).

\subsection{Double Fixed Points}
When defining the angle function $\Delta_i$ at a triple with $x_i=0$, we appealed to Proposition~\ref{prop:cycle nondegenerate} to conclude that if $N \geq 2$, then the $2N$ points \eqref{eq:cycle points} were all distinct. In other words if $a_i^2\neq 4$, then no triple with $x_i=0$ can be fixed by $m_{i\pm1}$ by Lemma~\ref{lemma:N=1 iff a_i^2=4}. We can say more in the reverse direction, that is describe the $i$th coordinate of the triples that are fixed under the action of $m_{i-1}$ and $m_{i+1}$ simultaneously.

\begin{proposition}
    Suppose $m_{i-1}x=m_{i+1}x=x$, then
    \begin{align*}
        0=x_i^2(u^2-4)(u^2+a_{i-1}a_{i+1}u+a_{i-1}^2+a_{i+1}^2-4)
    \end{align*}
    where $u=sx_i-a_i$.
\end{proposition}
\begin{proof}
    The condition that $x$ is fixed by both $m_{i-1}$ and $m_{i+1}$ can be written as  
        \begin{align*}
           \begin{cases}
               2x_{i-1}&=sx_ix_{i+1}-a_ix_{i+1}-a_{i+1}x_i = (sx_i-a_i)x_{i+1}-a_{i+1}x_i \\
               2x_{i+1}&=sx_ix_{i-1}-a_ix_{i-1}-a_{i-1}x_i = (sx_i-a_i)x_{i-1}-a_{i-1}x_i.
           \end{cases} 
        \end{align*}
    By setting $u=sx_i-a_i$, multiplying both equations by 2, and substituting each $2x_{i\pm 1}$ into the other equation, we arrive at
    \begin{align*}
           \begin{cases}
               4x_{i-1}& = 2ux_{i+1}-2a_{i+1}x_i=u^2x_{i-1}-a_{i-1}ux_{i}-2a_{i+1}x_i \\
               4x_{i+1}& = 2ux_{i-1}-2a_{i-1}x_i=u^2x_{i+1}-a_{i+1}ux_{i}-2a_{i-1}x_i.
           \end{cases} 
        \end{align*}
    Rearranging terms to isolate $x_{i-1}$ and $x_{i+1}$ respectively yields
    \begin{align*}
           \begin{cases}
               (u^2-4)x_{i-1}& = (a_{i-1}u+2a_{i+1})x_i \\
               (u^2-4)x_{i+1}& = (a_{i+1}u+2a_{i-1})x_i.
           \end{cases} 
        \end{align*}
    By subtracting $sx_1x_2x_3$ from both sides of Equation~(\ref{eqn:family}), multiplying by $(u^2-4)^2$, substituting in $(u^2-4)x_{i\pm1}=(a_{i\pm1}u+2a_{i\mp1})x_i$, and simplifying, we have
    \begin{align*}
        (u^2-4)^2&(x_1^2+x_2^2+x_3^2+a_1x_2x_3+a_2x_1x_3+a_3x_1x_2-sx_1x_2x_3) \\ 
        &= x_i^2(u^2-4)^2+((a_{i-1}u+2a_{i+1})x_i)^2+((a_{i+1}u+2a_{i-1})x_i)^2\\
        &\ \ \ \ +a_i((a_{i-1}u+2a_{i+1})x_i)((a_{i+1}u+2a_{i-1})x_i) \\
        &\ \ \ \ + (a_{i-1}x_i((a_{i+1}u+2a_{i-1})x_i) + a_{i+1}x_i((a_{i-1}u+2a_{i+1})x_i))(u^2-4)
        \\
        &\ \ \ \ -sx_i((a_{i-1}u+2a_{i+1})x_i)((a_{i+1}u+2a_{i-1})x_i) \\
        &= x_i^2(u^2-4)(u^2+a_{i-1}a_{i+1}u+a_{i-1}^2+a_{i+1}^2-4). \\
    \end{align*}
    Since this expression is 0 for any solution of Equation~(\ref{eqn:family}), we arrive at
        \begin{align}\label{eqn:doublefixed}
            x_i^2(u^2-4)(u^2+a_{i-1}a_{i+1}u+a_{i-1}^2+a_{i+1}^2-4) = 0.
        \end{align}
\end{proof}
By examining the factors of Equation~(\ref{eqn:doublefixed}), we see that the Cayley cubic also comes into play when evaluating double fixed points. If the factor $u^2+a_{i-1}a_{i+1}u+a_{i-1}^2+a_{i+1}^2-4$ vanishes, then $x$ is fixed by both $m_{i-1}$ and $m_{i+1}$, and $(-sx_i+a_i, a_{i-1}, a_{i+1})$ lies on Cayley's cubic.

\subsection{Example}

In the Markoff case $a_1=a_2=a_3=0$, we have $r^2+1=0$ so $N=2$. The cycles have length 4 and $\Delta_i(x)=0$ throughout the subset where $x_i=0$. If we start from $\Delta_{i-1}(x)=\delta$, then $\Delta_{i\pm 1}$ cycle through the values $\delta$ and $s-\delta$. A symmetrical choice $\delta=\frac{s}{2}$ makes these constant, which is the approach Martin used in \cite{Martin}. For other parameters $a_1$, $a_2$, $a_3$, if $\rho$ has a larger order $N$, it may not be possible to have a constant vector $(\Delta_{i-1},\Delta_{i+1})$ and it may not be possible to have $\Delta_{i-1}=\Delta_{i+1}$.
In general, a move $m_{i \pm 1}$ may change all three values $\Delta_{1,2,3}(x)$.
Instead of $s=s/2+s/2$ pointwise, there is a somewhat similar balance if we average both $i-1$ and $i+1$ over the whole cycle:
\[
\frac{1}{2N} \sum_{g \in \langle m_{i-1},m_{i+1}\rangle} \frac{\Delta_{i-1}+\Delta_{i+1}}{2}(gx) = s.
\]

\begin{figure}
\begin{tikzpicture}[scale=2]
\draw (0,0) node(0){$(\delta,s-\delta)$};
\draw (0)++(225:1) node(1){$(s-\delta,\delta)$};
\draw (0,0)++(225:1)++(-45:1) node(2){$(\delta,s-\delta)$};
\draw (0)++(-45:1) node(3){$(s-\delta,\delta)$};
\draw (0)--(1)--(2)--(3)--(0);
\end{tikzpicture}
\caption{
Values of $(\Delta_{i-1},\Delta_{i+1})$ in the Markoff case ($a_1=a_2=a_3=0$, all $\Delta_i=0$ on $\{x_i=0\}$, and $r=\sqrt{-1}$), where $\delta$ is an arbitrary value of $\Delta_{i-1}(x)$ to start the cycle.
}
\end{figure}
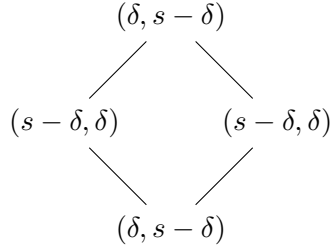

\section{Quadratic obstruction: proof of Theorem~\ref{thm:breakup}}
\label{section:quadobstruction}
In this section, we show there are at least two orbits for parameters of the special form $(2\sigma,\pa,\pa\sigma)$ where $\sigma=\pm 1$ and $\pa \in \F_p$ (Theorem~\ref{thm:breakup}).
The same method shows that, specialising further to $\pa=\pm 2$, there are at least four orbits.
To simplify the notation, let us rescale so that
\[
s=1.
\]
We choose the indices so that $a_i=2\sigma$, $a_{i+1}=\pa$, and $a_{i-1}=\pa\sigma$.

The proof uses both of Vieta's rules for the roots of a quadratic equation.
If $x$ solves (\ref{eqn:family}), then
\begin{align}
x_i+x_i' &= sx_{i-1}x_{i+1} - a_{i-1}x_{i+1}-a_{i+1}x_{i-1} \label{vieta-sum} \\
x_i x_i' &= x_{i-1}^2+x_{i+1}^2+a_i x_{i-1}x_{i+1} \label{vieta-product}.
\end{align}
If $a_i = 2\sigma$, then we have a perfect square
\[
x_ix_i' = (x_{i-1} + \sigma x_{i+1})^2
\]
so there is a quadratic obstruction:
\begin{equation} \label{eqn:obstruction}
\chi(x_i)\chi(x_i') \neq -1.
\end{equation}
The equation (\ref{eqn:family}) can also be written as
\begin{align*}
(x_i+ x_{i-1}+\sigma x_{i+1})^2 &= x_i(x_{i-1}x_{i+1}-\pa x_{i-1}-\pa \sigma x_{i+1}+2x_{i-1}+2\sigma x_{i+1}) \\
&=x_i \big(x_i+x_i' + 2(x_{i-1}+\sigma x_{i+1}) \big)
\end{align*}
so
\begin{equation} \label{eqn:obstruction2}
\chi(x_i) \chi( x_i+x_i' + 2x_{i-1}+2\sigma x_{i+1}) \neq -1.
\end{equation}

Therefore $\chi(x_i)$ and $\chi(x_i+x_i'+2x_{i-1}+2\sigma x_{i+1})$ are either both $\leq 0$ or both $\geq 0$.
We claim that the moves $m_1$, $m_2$, $m_3$ preserve each of the corresponding subsets, which combined with the fact that both $\chi(x_i)$ and $\chi(x_i+x_i'+2x_{i-1}+2\sigma x_{i+1})$ cannot be simultaneously 0 for $x \neq (0,0,0)$, implies Theorem~\ref{thm:breakup}.
The subsets may decompose further, as indeed they do if $\pa= \pm 2$.

To see that $\chi(x_i)$ and $\chi(x_i+x_i'+2x_{i-1}+2\sigma x_{i+1})$ cannot simultaneously be equal to 0, suppose $\chi(x_i)=0$. This implies that $x_i=0$. Plugging $x_i=0$ into equation~(\ref{eqn:family}) produces $x_{i-1}+\sigma x_{i+1}=0$. If we were to have $\chi(x_i+x_i'+2x_{i-1}+2\sigma x_{i+1})=0$ as well, then $x_i'=0$. Plugging $x_i=x_i'=x_{i-1}+\sigma x_{i+1}=0$ into the Vieta involution~(\ref{vieta-sum}) gives $sx_{i-1}x_{i+1}=0$, so either $x_{i-1}=0$ or $x_{i+1}=0$. In either case, we arrive at the only point with $\chi(x_i)=\chi(x_i+x_i'+2x_{i-1}+2\sigma x_{i+1})=0$ is $x=(0,0,0)$.

The claim is clear for the involution $m_i$ because $x_i \mapsto x_i'$ leaves $x_i+x_i'+2x_{i-1}+2\sigma x_{i+1}$ invariant. Meanwhile (\ref{eqn:obstruction}) shows that $\chi(x_i)$ and $\chi(x_i')$ cannot differ by a sign (although either might be 0).

The claim holds for $m_{i\pm 1}$ as well after further calculation.
For $m_{i+1}$, one shows that
\begin{equation}
  \begin{split}
&(x_{i-1}x_{i+1} -\sigma(\pa-2)x_{i+1}-(\pa-2) x_{i-1}) \\ 
&\quad\cdot (x_{i-1}x_{i+1}' -\sigma(\pa-2)x_{i+1}'-(\pa-2) x_{i-1}) \\
&\quad\quad = (x_{i-1}^2+x_{i-1}x_i - \sigma(\pa-2) x_i )^2 \label{perfect-square}
  \end{split}
\end{equation}
is a square.
A similar identity holds for $i-1$, and can be obtained automatically by considering $-\pa$ instead of $\pa$ if $\sigma=-1$.
\begin{proof}[Proof of (\ref{perfect-square})]
Starting from the top, the idea is to collect terms $x_{i+1}x_{i+1}'$ or $x_{i+1}+x_{i+1}'$ so that Vieta's rule can be applied, as in (\ref{vieta-sum}) and (\ref{vieta-product}) for $x_{i+1}$ instead of $x_i$.
That gives
\begin{align*}
&x_{i+1}x_{i+1}' \big(x_{i-1}-(\pa-2)\sigma\big)^2 \\
+&(x_{i+1}+x_{i+1}') (\pa-2)\big(-x_{i-1}+(\pa-2)\sigma\big)x_{i-1} \\
+&(\pa-2)^2x_{i-1}^2 .
\end{align*}
After substituting Vieta's rules
\begin{align*}
x_{i+1}x_{i+1}' &= \pa x_{i-1}x_i + x_{i-1}^2+x_i^2 \\
x_{i+1}+x_{i+1}' &= x_{i-1}x_i-\sigma \pa x_i - 2\sigma x_{i-1},
\end{align*}
several terms cancel. In particular, the coefficient of $x_{i-1}^2$ is
\[
(\pa-2)^2+(\pa-2)^2-2(\pa-2)^2 = 0,
\]
the coefficient of $x_{i-1}^3$ is
\[
-2\sigma (\pa-2)+2\sigma (\pa-2) = 0,
\]
and the coefficient of $x_{i-1}x_i$ is
\[
\pa(\pa-2)^2 - \pa(\pa-2)^2 = 0.
\]
From the remaining terms, one has
\begin{align*}
&(x_{i-1}x_{i+1} -\sigma(\pa-2)x_{i+1}-(\pa-2) x_{i-1}) \\ &\quad \cdot
(x_{i-1}x_{i+1}' -\sigma(\pa-2)x_{i+1}'-(\pa-2) x_{i-1}) \\
&= x_{i-1}^4 + 2x_{i-1}^3 x_i + x_{i-1}^2x_i^2 \\ &\quad - 2\sigma(\pa-2)x_{i-1}^2 x_i - 2\sigma(\pa-2) x_{i-1} x_i^2 + (\pa-2)^2 x_i^2 \\
&= (x_{i-1}^2 + x_{i-1} x_i - \sigma (\pa-2) x_i )^2
\end{align*}
as required.
\end{proof}

One can also start from the other side of (\ref{perfect-square}) with
\[
x_{i-1}^2+x_{i-1}x_i - \sigma(\pa-2) x_i = x_{i-1}^2+x_{i+1}+x_{i+1}'+2\sigma(x_i-x_{i-1}).
\]
Squaring this gives another approach to proving the identity.

\subsection{Further splitting if $\pa=\pm 2$}

In the most degenerate case, (\ref{eqn:family}) becomes
\[
(x \pm y + z)^2 = s xyz
\]
so
\[
\chi(x_1)\chi(x_2)\chi(x_3) \neq -\chi(s).
\]
The moves preserve four subsets where any two characters assume a given sign.

\section{Number of solutions mod $p$}

In this section, we compute the number of solutions to (\ref{eqn:family}) over $\F_p$ for a prime $p \neq 2$.
Similar calculations were done by Carlitz \cite{Carlitz}, who considered varying the coefficients in $x_1^2+x_2^2+x_3^2$ instead of adding off-diagonal terms.
Recall that $\chi$ denotes the quadratic character modulo $p$.
The number of solutions depends on the three values $\chi(a_i^2-4)$ as well as whether the parameters lie on the surface
\[
a_1^2+a_2^2+a_3^2=a_1a_2a_3+4.
\]

\begin{proposition} \label{prop:numel}
The number of solutions to (\ref{eqn:family}) modulo $p$ excluding $(0,0,0)$ is
\begin{equation} \label{eqn:numel2}
    p^2+ p\Big(\sum_{i=1}^3 \chi(a_i^2-4)+C(a_1,a_2,a_3) \Big)
\end{equation}
where $C: \F_p^3 \rightarrow \{0,1,-1\}$ is given by
\begin{align*}
C(a_1,a_2,a_3)=
\begin{cases}
0
& \text{if } a_1^2+a_2^2+a_3^2\neq a_1a_2a_3+4, \\[2mm]
-\chi(\pa^2-4)
& \begin{aligned}[t]
  &\text{if } a_i = 2\sigma,\ a_{i-1}=\pa,\ \text{and } a_{i+1}=\pa \\
  &\text{for some } i,\sigma=\pm 1,\text{ and }\pa,
  \end{aligned} \\[2mm]
\displaystyle -\prod_i \chi(a_i^2-4)
& \text{otherwise.}
\end{cases}
\end{align*}

\end{proposition}
The function $C$ tells us that the number of solutions is $p^2+np$ where $|n|\leq 3$.
See \cite{Swinn} for the extreme values of $n$ among cubic surfaces.

The proof amounts to understanding conic sections over $\F_p$. 
Fixing one coordinate in (\ref{eqn:family}) defines a curve in the other two, which is a conic of the form
\begin{equation} \label{eqn:conic-bdef}
    x^2+y^2+Bxy+Dx+Ey+F=0.
\end{equation}
For a given value of $x_i$ and $(x,y)=(x_{i-1},x_{i+1})$, the parameters are
\begin{equation} \label{eqn:bdef}
B=a_i-sx_i, \quad D=a_{i+1}x_i, \quad E=a_{i-1}x_i, \quad F = x_i^2.
\end{equation}
In most cases, the number of solutions is $p-\chi(B^2-4)$, that is $p+1$ points on an ellipse, $p-1$ points on a hyperbola, or $p$ points on a parabola.
However, several other possibilities occur for special values of $B,D,E,F$.
There are $2p$ points on a pair of parallel lines, $2p-1$ points on a pair of intersecting lines, $p$ points on two copies of a single line, just $1$ point on a pair of ``imaginary" lines intersecting at a single ``real" point, and $0$ points on a pair of imaginary parallel lines.
To determine the outcome, we must put conics into a standard form by completing the square.

It is useful to note the following fact about shifted squares (which was also known to Carlitz).
For any $c \neq 0$ in $\F_p$, it follows from Euler's criterion that
\begin{equation} \label{eqn:shifted-squares}
\sum_{t \in \F_p} \chi(t^2-c) = -1
\end{equation}
where $\chi$ is the quadratic character.

\begin{proposition} \label{prop:count-conic}
For $\alpha \neq 0$ and $\beta \neq 0$ in $\F_p$, the number of solutions to $x^2-\alpha y^2 = \beta$ is
\[
p - \chi(\alpha).
\]
\end{proposition}
\begin{proof}
For each $y$, there are $1+\chi(\beta+\alpha y^2)$ solutions to $x^2=\beta + \alpha y^2$.
Summing over $y$ gives the total as
\[
\sum_y \big( 1 + \chi(\beta + \alpha y^2) \big) = p + \chi(\alpha) \sum_y \chi(y^2 + \alpha^{-1} \beta)
\]
which is $p-\chi(\alpha)$ by (\ref{eqn:shifted-squares}).
\end{proof}

\subsection{Completing the square}

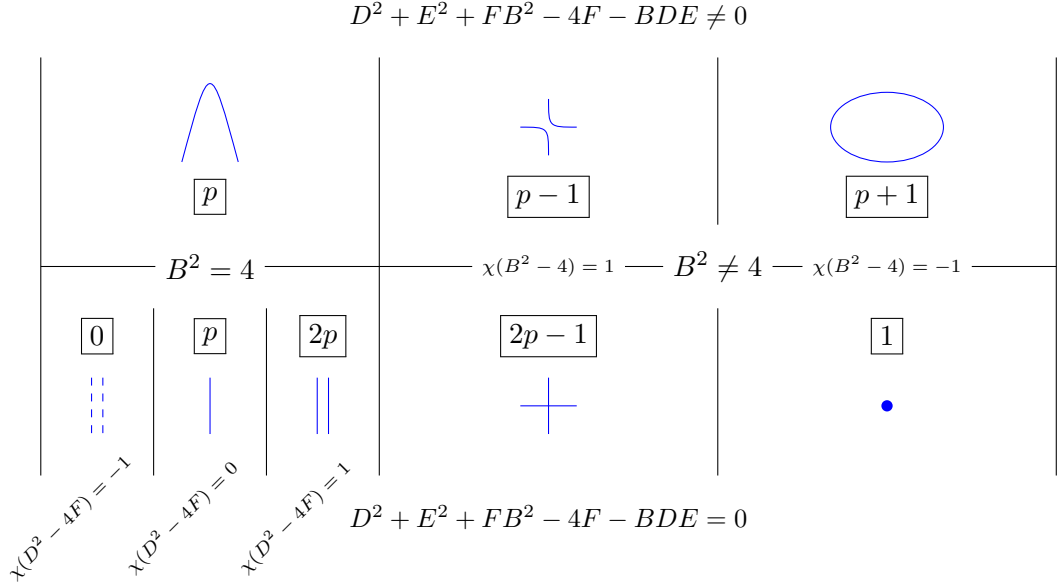
\begin{figure}
\centering
\begin{tikzpicture}[scale=1.5]
\pgfmathsetmacro{\h}{0.618}
\pgfmathsetmacro{\c}{0.6*\h}
\pgfmathsetmacro{\g}{3.6*\h}
\pgfmathsetmacro{\x}{-0.25}
\pgfmathsetmacro{\s}{0.05}
\draw (0,-3*\h)--(0,3*\h);
\draw (3,-3*\h)--(3,3*\h);
\draw (9,-3*\h)--(9,3*\h);
\draw (1,-\c)--(1,-3*\h);
\draw (2,-\c)--(2,-3*\h);
\draw (6,\c)--(6,3*\h);
\draw (6,-\c)--(6,-3*\h);
\draw (4.5,\g) node{{\small $D^2+E^2+FB^2-4F-BDE \neq 0$}};
\draw (4.5,-\g) node{{\small $D^2+E^2+FB^2-4F-BDE = 0$}};
\draw (1.5,0) node(4){$B^2 = 4$};
\draw (6,0) node(not4){$B^2 \neq 4$};
\draw (4.5,0) node(chi+){{\tiny $\chi(B^2-4)=1$}};
\draw (7.5,0) node(chi-){{\tiny $\chi(B^2-4)=-1$}};
\draw (0,0)--(4)--(chi+);
\draw (chi-)--(9,0);
\draw (chi+)--(not4)--(chi-);
\draw (0.5+\x,-\g) node[rotate=45]{{\tiny $\chi(D^2-4F)=-1$}};
\draw (1.5+\x,-\g) node[rotate=45]{{\tiny $\chi(D^2-4F)=0$}};
\draw (2.5+\x,-\g) node[rotate=45]{{\tiny $\chi(D^2-4F)=1$}};
\draw (1.5,\h) node{\fbox{$p$}};
\draw (4.5,\h) node{\fbox{$p-1$}};
\draw (7.5,\h) node{\fbox{$p+1$}};
\draw (0.5,-\h) node{\fbox{$0$}};
\draw (1.5,-\h) node{\fbox{$p$}};
\draw (2.5,-\h) node{\fbox{$2p$}};
\draw (4.5,-\h) node{\fbox{$2p-1$}};
\draw (7.5,-\h) node{\fbox{$1$}};
%
\draw[blue] (7.5,2*\h) ellipse(0.5 and 0.309);
\fill[blue] (7.5,-2*\h) ellipse (0.05 and 0.05);
\draw[blue] (4.5-\x,2*\h) .. controls (4.5,2*\h) .. (4.5,2*\h-\x);
\draw[blue] (4.5+\x,2*\h) .. controls (4.5,2*\h) .. (4.5,2*\h+\x);
\draw[blue] (4.5-\x,-2*\h)--(4.5+\x,-2*\h);
\draw[blue] (4.5,-2*\h-\x)--(4.5,-2*\h+\x);
\draw[blue] (1.5-\x,1.5*\h) .. controls (1.5,3*\h) .. (1.5+\x,1.5*\h);
\draw[blue] (1.5,-2*\h-\x)--(1.5,-2*\h+\x);
\draw[blue] (2.5-\s,-2*\h-\x)--(2.5-\s,-2*\h+\x);
\draw[blue] (2.5+\s,-2*\h-\x)--(2.5+\s,-2*\h+\x);
\draw[blue,dashed] (0.5-\s,-2*\h-\x)--(0.5-\s,-2*\h+\x);
\draw[blue,dashed] (0.5+\s,-2*\h-\x)--(0.5+\s,-2*\h+\x);
\end{tikzpicture}
\caption{
The number of solutions to $x^2+Bxy+y^2+Dx+Ey+F=0$ in $\F_p$.
Top: the number is $p-\chi(B^2-4)$ for smooth conics, in terms of the quadratic character $\chi \bmod p$.
Bottom: if $D^2+E^2+FB^2-4F-BDE = 0$, there is a correction of $p$ times $\chi(B^2-4)$ or $\chi(D^2-4F)$.
Blue: analogous conic sections over the reals.
}
\label{fig:conics}
\end{figure}

Let us complete the square, starting from
\[
x^2+Bxy+y^2+Dx+Ey+F = 0.
\]
Assuming $B^2 \neq 4$, we get
\[
0= \Big(x + \frac{B}{2}y + \frac{D}{2} \Big)^2 + \frac{4-B^2}{4} \Big(y + \frac{2E-BD}{4-B^2} \Big)^2 + F - \frac{D^2}{4} - \frac{1}{4} \frac{(2E-BD)^2}{4-B^2}.
\]
After some simplifications, the conic becomes
\begin{equation}
  \begin{split}
&(2x+By+D)^2 - (B^2-4) \Big( y + \frac{2E-BD}{4-B^2} \Big)^2 = \\
&\frac{4(FB^2+D^2+E^2-EBD-4F)}{4-B^2}.
  \end{split} \label{eqn:conic}
\end{equation}
Assuming the constant on the right-hand side of \eqref{eqn:conic} is non-zero, Proposition~\ref{prop:count-conic} shows that the number of solutions is
\[
p-\chi(B^2-4).
\]
If the constant vanishes, instead the conic is a pair of lines (possibly imaginary, and not necessarily distinct).
The lines are given by
\[
2x+By+D = \pm \sqrt{B^2-4} \Big( y + \frac{2E-BD}{4-B^2} \Big).
\]
They intersect in a point with coordinates
\[
y= \frac{BD-2E}{4-B^2}, \quad x = -\frac{1}{2}\Big(D + \frac{B(BD-2E)}{4-B^2}\Big) = \frac{BE-2D}{4-B^2}.
\]
If $\chi(B^2-4)=-1$, then this intersection point is the only solution in $\F_p$.
If $\chi(B^2-4)=1$, then there are $2p-1$ solutions.

Assuming instead that $B^2=4$, completing the square gives
\[
(2x+By+D)^2 + 4\Big(E - \frac{B}{2}D\Big)y = D^2-4F.
\]
If $E-BD/2 \neq 0$, then this defines a parabola with $p$ points where any value of the linear form $2x+By+D$ gives a unique value for $y$.
This total $p$ is again of the form $p-\chi(B^2-4)$ in the special case $B^2=4$.
If $E-BD/2 =0$, then the conic degenerates to
\[
(2x+By+D)^2 = D^2-4F
\]
which is a line if $D^2-4F=0$ or two parallel lines if $D^2-4F \neq 0$.
In the latter case, there are either $0$ solutions over $\F_p$ if $\chi(D^2-4F)=-1$ or $2p$ solutions if $\chi(D^2-4F)=1.$
All three counts can be written together as $p+\chi(D^2-4F)p.$

The cases $B^2=4$ and $B^2 \neq 4$ have something in common: if $B^2 =4$, then
\[
D^2+E^2+FB^2-4F-BDE = \big(E - \frac{B}{2}D\big)^2
\]
so the further case distinction can be seen in a unified way.
The number of points on the conic is
\[
p-\chi(B^2-4) + \mathbbm{1}[D^2+E^2+FB^2-4F-BDE=0] \cdot p \cdot \begin{cases}\chi(B^2-4) \\ \chi(D^2-4F) \end{cases}
\]
where the case-wise final term is $\chi(D^2-4F)$ if $B^2-4 = 0$ and $E=\frac{B}{2}D$, and $\chi(B^2-4)$ otherwise.

If $F=1$, then $D^2+E^2+B^2-4-BDE$ is of course Cayley's cubic.

\begin{proof}[Proof of Proposition~\ref{prop:numel}]
The total is
\[
\# (x_1,x_2,x_3) = \sum_{x_3} \#(x_1,x_2)
\]
where for each $x_3$, we count $(x_1,x_2)$ on a conic of the form (\ref{eqn:conic-bdef}) with the parameters from (\ref{eqn:bdef}).
Explicitly, (\ref{eqn:family}) becomes
\[
x_1^2 + (a_3-sx_3)x_1x_2 + x_2^2 + a_2x_3 x_1 + a_1x_3x_2 + x_3^2 = 0
\]
with $B=a_3-sx_3$. For any $s \neq 0$, we may as well change variables from $x_3$ to $B$.
The number of solutions is therefore
\[
\sum_{B \in \F_p} \big( p-\chi(B^2-4) \big) + p \sideset{}{'}\sum_B \begin{cases}\chi(B^2-4) \\ \chi(D^2-4F) \end{cases}
\]
where $\sum_B^{'}$ runs over solutions to $D^2+E^2+FB^2-4F-BDE=0$.

The first sum gives
\[
\sum_{B \in \F_p} \big( p-\chi(B^2-4) \big) = p^2 + 1
\]
by (\ref{eqn:shifted-squares}), which becomes $p^2$ after discounting $(0,0,0)$.

In the second sum, we have
\[
D^2+E^2+FB^2-4F-BDE=x_3^2 \big(sx_3 (sx_3 + a_{1}a_{2}-2a_3) + a_1^2+a_2^2+a_3^2-a_1a_2a_3-4 \big)
\]
where $x_3=0$ for $B=a_3$. The second factor is
\[
\big(B-a_3\big)\big(B-(a_{1}a_{2}-a_3)\big)+a_1^2+a_2^2+a_3^2-a_1a_2a_3-4
\]
which vanishes if
\[
B^2 - a_{1}a_{2}B + a_{1}^2+a_{2}^2-4 = 0.
\]
That leaves up to three terms in $\sideset{}{'}\sum_B$,
namely
\[
B=a_3, \quad B=t_+, \quad B=t_-
\]
where
\[
t_{\pm} = \frac{a_1a_2 \pm \sqrt{a_1^2a_2^2 - 4(a_1^2+a_2^2-4)}}{2} = \frac{a_1a_2 \pm \sqrt{(a_1^2-4)(a_2^2-4)}}{2}
\]
which must be excluded if there is no such square root in $\F_p$.
This gives a variant of Proposition~\ref{prop:numel} with
\[
C(a_1,a_2,a_3)= -\sum_i \chi(a_i^2-4) + \sum_{B \in \{a_3,t_+,t_-\} \ \cap \ \F_p } \begin{cases} \chi(B^2-4) \\ \chi(D^2-4F) \end{cases}
\]
which simplifies as claimed in a case-by-case manner depending on the sign of $\chi(a_1^2-4)\chi(a_2^2-4)$. We note here that the degenerate case of $a_3=t_{\pm}$ is equivalent to the statement that $(a_1,a_2,a_3)$ lies on the Cayley cubic to highlight its role in the overall point count for (\ref{eqn:family}).

It remains to determine whether there are one, two, or three terms and whether each contributes $\chi(B^2-4)$ or $\chi(D^2-4F)$.
Each term contributes $\chi(B^2-4)$ unless $B^2=4$.
If $B^2=4$, then
\[
0=D^2+E^2+F(B^2-4)-BDE \implies E-BD/2=0.
\]
With $D=a_1 x_3$ and $E=a_2 x_3$, it must be that either $a_1=\pm a_2$ or $x_3=0$.
We have $x_3=0$ if and only if $B=a_3$, so this is only possible when $a_3^2=4$.
Thus all summands are $\chi(B^2-4)$ for generic parameters $(a_1,a_2,a_3)$.
Let us first write the proof assuming all summands are $\chi(B^2-4)$ and then indicate how to include possible terms $\chi(D^2-4F)$.

We will use one more observation to simplify the final answer.
The equation $t^2-a_1a_2t + a_1^2+a_2^2-4=0$ can equally well be solved for $a_1$ or $a_2$ in terms of either of the roots $t_{\pm}$.
Whereas $t_{\pm}$ could lie in an extension, we know $a_1$ and $a_2$ lie in $\F_p$.
It follows that
\begin{equation} \label{eqn:symmetriser}
\chi(a_1^2-4)\chi(t^2-4) \neq -1,  \quad \chi(a_2^2-4)\chi(t^2-4) \neq -1
\end{equation}
which will give the formula as claimed with symmetry between $(a_1,a_2,a_3)$ restored.

There are three cases depending on whether $\chi(a_1^2-4)\chi(a_2^2-4)$ is positive, negative, or zero.
Consider first the case
\begin{equation} \label{eqn:case-}
\chi\big( (a_1^2-4)(a_2^2-4) \big) = -1
\end{equation}
hence there is only a single term $B=a_3$ in $\sideset{}{'}\sum_B$.
If it contributes $\chi(B^2-4)$, rather than $\chi(D^2-4F)$, then the total is
\[
\chi(a_3^2-4) = \sum_i \chi(a_i^2-4)
\]
because $\chi(a_1^2-4)+\chi(a_2^2-4)=0$.
From the sign condition (\ref{eqn:case-}), it also follows that $(a_1,a_2,a_3)$ is not on the Cayley cubic.
Thus the proposition holds in this case.

Next consider
\begin{equation} \label{eqn:case0}
    (a_1^2-4)(a_2^2-4)=0.
\end{equation}
In this case,
\[
t_+ = t_- = \frac{a_1a_2}{2}
\]
so there could be two values if $a_3 \neq \frac{a_1a_2}{2}$ or a single value if $(a_1,a_2,a_3)$ lies on the Cayley cubic.
If there are two values, then
\[
\sideset{}{'}\sum_B \chi(B^2-4) = \chi(a_3^2-4) + \chi\Big(\frac{a_1^2a_2^2}{4}-4\Big).
\]
Let us say $a_1^2=4$, the argument being the same if $a_2^2=4$.
Then $\frac{a_1^2a_2^2}{4}-4=a_2^2-4$ and $\chi(a_1^2-4)=0$ so
\[
\sideset{}{'}\sum_B \chi(B^2-4) = \chi(a_3^2-4) + \chi\Big(\frac{a_1^2a_2^2}{4}-4\Big) = \sum_i \chi(a_i^2-4).
\]
This proves the proposition assuming $a_3 \neq \frac{a_1a_2}{2}$.
If $a_3 = \frac{a_1a_2}{2}$, then there is just one term, and
\[
\sideset{}{'}\sum_B \chi(B^2-4) = \chi(a_3^2-4) = -\chi(a_2^2-4)+\sum_i \chi(a_i^2-4)
\]
which also agrees with the proposition.

Finally, suppose
\begin{equation} \label{eqn:case+}
\chi\big( (a_1^2-4)(a_2^2-4) \big) = 1.
\end{equation}
The sum either involves three distinct values $a_3, t_+, t_-$ or, in the Cayley case, two values $a_3$ and $a_1a_2-a_3$.
Suppose there are three.
Using (\ref{eqn:symmetriser}), and assuming $a_i^2-4$ and $t_{\pm}^2-4$ do not vanish, we conclude that $\chi(a_1^2-4)$ and $\chi(t^2-4)$ have the same sign (and similarly for $a_2$).
Therefore
\[
\sideset{}{'}\sum_B \chi(B^2-4) = \chi(a_3^2-4) + \chi(t_+^2-4)+\chi(t_-^2-4) = \sum_{i} \chi(a_i^2-4)
\]
as required.
If there are only two terms, then using (\ref{eqn:symmetriser}) as before, we find
\[
\sideset{}{'}\sum_B \chi(B^2-4) = \chi(a_3^2-4)+\chi(t^2-4)=-\chi(a_1^2-4)+\sum_i \chi(a_i^2-4)
\]
where the correction is again as claimed.

To complete the proof, it remains only to consider the possibility of summands $\chi(D^2-4F)$ instead of $\chi(B^2-4)$.
These arise only when $B^2=4$ and $E= \frac{B}{2}D$, that is,
\[
a_2 x_3 = \frac{B}{2} a_1 x_3 = \pm a_1 x_3.
\]
If $x_3=0$, then $B=a_3-sx_3=a_3$ so $a_3^2=4$ and $D=E=F=0$ meaning $\chi(B^2-4)=\chi(D^2-4F)=0$. This shows that either summand gives the same result for $B=a_3$.

If $B^2=4$ for one of the other possible terms $B=t_{\pm}$, then $x_3 \neq 0$ so $a_2 = \frac{B}{2}a_1$. Therefore
\[
D^2-4F = (a_1^2-4) \Big( \frac{B-a_3}{s} \Big)^2.
\]
It follows that
\[
\chi(D^2-4F) = \begin{cases} \chi(a_1^2-4) &\text{if} \ B \neq a_3 \\
0 &\text{if} \ B=a_3
\end{cases}
\]
and likewise for $\chi(a_2^2-4)$ since $a_1^2=a_2^2$ in this scenario.

The proposition follows as before by considering the different sign possibilities (\ref{eqn:case-}), (\ref{eqn:case0}), and (\ref{eqn:case+}).
For example, in case (\ref{eqn:case+}),
\[
\sideset{}{'}\sum_B \chi(D^2-4F) = \sum_i \chi(a_i^2-4)
\]
with $\chi(a_3^2-4)$ from $B=a_3$ and the remaining terms from $B=t_{\pm}$.
\end{proof}

\section{Counterexamples} \label{sec:counter}

In this section, we illustrate how divisibility by $p$ can fail if the hypotheses of Theorem~\ref{thm:congruence-lower-bound} are not satisfied.
If $s=0$, then (\ref{eqn:family}) may lead to highly disconnected graphs with many orbits of size not divisible by $p$.
Likewise for parameters such as $(2,2,-2)$ violating (\ref{eqn:hypo}) there are many small orbits as shown in Table~\ref{table:22-2}.

What makes the proofs break down in these examples is the presence of double fixed points $x=m_{i-1}x=m_{i+1}x$ at some $x$ with $x_i=0$.
At such a triple, (\ref{eqn:delta-pair}) forces $\Delta_{i-1}(x)=\Delta_{i+1}(x)=s/2$ and then (\ref{eqn:delta-total}) forces $\Delta_i(x) = 0$.
This may turn out to be inconsistent with other instances of (\ref{eqn:delta-total}) and (\ref{eqn:delta-pair}) at nearby points, so that there is no well-defined extension of the angle functions from (\ref{eqn:delta}) to triples with vanishing coordinates.

\subsection{Example: $(2,2,-2)$}

The sign condition (\ref{eqn:hypo}) is necessary in Theorem~\ref{thm:congruence-lower-bound}.
If $s \neq 0$ and $a_1=a_2=2$ but $a_3=-2$, then there are orbits of size 1, orbits of size 2, orbits of size 4, and indeed many other sizes. For singleton orbits, there are always three given by $(4s^{-1},4s^{-1},0)$, $(4s^{-1},0,-4s^{-1})$, and $(0,4s^{-1},-4s^{-1})$. In the case of orbits of size 2, there are three ``barbells" formed by two double fixed points connected by a move on the non-fixed coordinate as seen in Figure~\ref{fig:size2}. 

\begin{figure}[h]
\centering
    \begin{tikzpicture}
\pgfmathsetmacro{\r}{2.5}
\draw (0,0) node(1){$(0,2s^{-1},-2s^{-1})$};
\draw (4,0) node(2){$(4s^{-1},2s^{-1},-2s^{-1})$};
\draw (0,-1) node(3){$(2s^{-1},0,-2s^{-1})$};
\draw (4,-1) node(4){$(2s^{-1},4s^{-1},-2s^{-1})$};
\draw (0.14,-2) node(5){$(2s^{-1},2s^{-1},0)$};
\draw (4,-2) node(6){$(2s^{-1},2s^{-1},-4s^{-1})$};
\filldraw[black] (6.6,0) circle (0pt) node[anchor=west]{\tiny $m_2,m_3$};
\filldraw[black] (6.6,-1) circle (0pt) node[anchor=west]{\tiny $m_1,m_3$};
\filldraw[black] (6.6,-2) circle (0pt) node[anchor=west]{\tiny $m_1,m_2$};
\filldraw[black] (-3.5,0) circle (0pt) node[anchor=west]{\tiny $m_2,m_3$};
\filldraw[black] (-3.5,-1) circle (0pt) node[anchor=west]{\tiny $m_1,m_3$};
\filldraw[black] (-3.2,-2) circle (0pt) node[anchor=west]{\tiny $m_1,m_2$};
\draw (1) edge["\tiny $m_1$"] (2);
\draw (3) edge["\tiny $m_2$"] (4);
\draw (5) edge["\tiny $m_3$"] (6);
\draw (-1.47,0.1) arc (15:345:.4);
\draw (-1.47,-0.9) arc (15:345:.4);
\draw (-1.17,-1.9) arc (15:345:.4);
\draw (5.76,0.1) arc (165:-165:.4);
\draw (5.76,-0.9) arc (165:-165:.4);
\draw (5.76,-1.9) arc (165:-165:.4);
\end{tikzpicture}
\caption{Orbits of size 2 for $a=(2,2,-2)$ and $s \neq 0$.}
\label{fig:size2}
\end{figure}
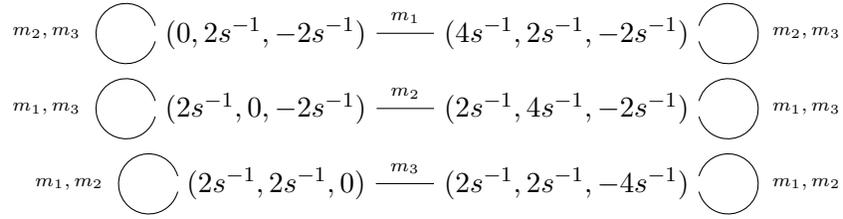

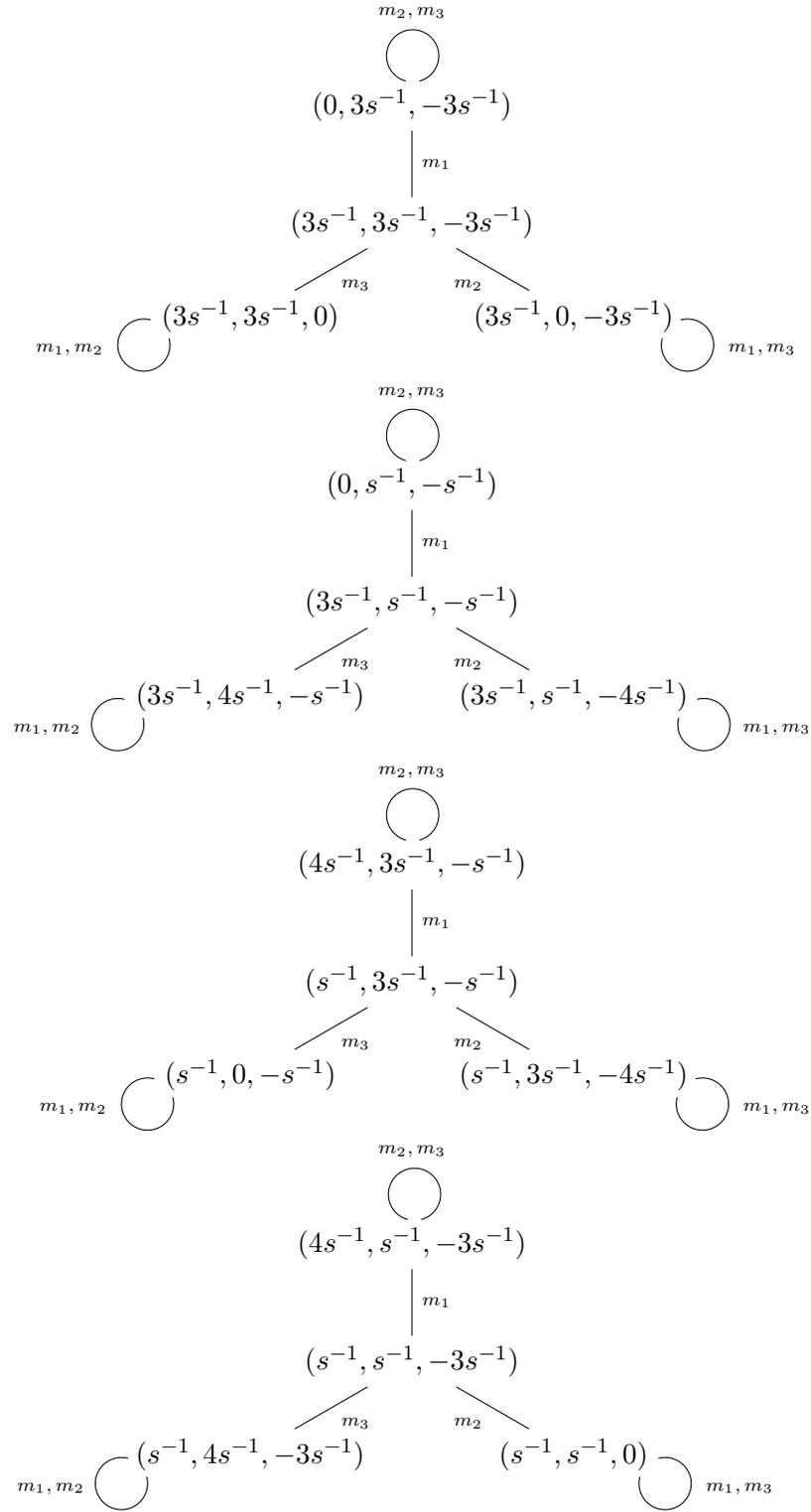
\begin{figure}
\centering 
\begin{tikzpicture}
\draw (0,0)++(90:1.58) node(2){$(0,3s^{-1},-3s^{-1})$};
\draw (0,0) node(1){$(3s^{-1},3s^{-1},-3s^{-1})$};
\draw (0,0)++(210:2.5) node(3){$(3s^{-1},3s^{-1},0)$};
\draw (0,0)++(330:2.5) node(4){$(3s^{-1},0,-3s^{-1})$};
\filldraw[black] (0,2.6) circle (0pt) node[anchor=south]{\tiny $m_2,m_3$};
\filldraw[black] (4.1,-1.7) circle (0pt) node[anchor=west]{\tiny $m_1,m_3$};
\filldraw[black] (-4,-1.7) circle (0pt) node[anchor=east]{\tiny $m_1,m_2$};
\draw (2) edge["\tiny $m_1$"] (1);
\draw (1) edge["\tiny $m_3$"] (3);
\draw (4) edge["\tiny $m_2$"] (1);
\draw (0.1,1.9) arc (-75:255:.35);
\draw (-3.5,-1.3) arc (75:375:.35);
\draw (3.6,-1.3) arc (105:-195:.35);
\filldraw[white] (6,0) circle (0pt);
\filldraw[white] (-6,0) circle (0pt);
    \end{tikzpicture}

    \begin{tikzpicture}
\draw (0,0)++(90:1.58) node(2){$(0,s^{-1},-s^{-1})$};
\draw (0,0) node(1){$(3s^{-1},s^{-1},-s^{-1})$};
\draw (0,0)++(210:2.5) node(3){$(3s^{-1},s^{-1},-4s^{-1})$};
\draw (0,0)++(330:2.5) node(4){$(3s^{-1},4s^{-1},-s^{-1})$};
\filldraw[black] (0,2.6) circle (0pt) node[anchor=south]{\tiny $m_2,m_3$};
\filldraw[black] (4.3,-1.7) circle (0pt) node[anchor=west]{\tiny $m_1,m_3$};
\filldraw[black] (-4.3,-1.7) circle (0pt) node[anchor=east]{\tiny $m_1,m_2$};
\draw (2) edge["\tiny $m_1$"] (1);
\draw (1) edge["\tiny $m_3$"] (3);
\draw (4) edge["\tiny $m_2$"] (1);
\draw (0.1,1.9) arc (-75:255:.35);
\draw (-3.85,-1.3) arc (75:375:.35);
\draw (3.82,-1.3) arc (105:-195:.35);
\filldraw[white] (6,0) circle (0pt);
\filldraw[white] (-6,0) circle (0pt);
    \end{tikzpicture}

    \begin{tikzpicture}
\draw (0,0)++(90:1.58) node(2){$(4s^{-1},3s^{-1},-s^{-1})$};
\draw (0,0) node(1){$(s^{-1},3s^{-1},-s^{-1})$};
\draw (0,0)++(210:2.5) node(3){$(s^{-1},3s^{-1},-4s^{-1})$};
\draw (0,0)++(330:2.5) node(4){$(s^{-1},0,-s^{-1})$};
\filldraw[black] (0,2.6) circle (0pt) node[anchor=south]{\tiny $m_2,m_3$};
\filldraw[black] (3.95,-1.7) circle (0pt) node[anchor=west]{\tiny $m_1,m_3$};
\filldraw[black] (-4.3,-1.7) circle (0pt) node[anchor=east]{\tiny $m_1,m_2$};
\draw (2) edge["\tiny $m_1$"] (1);
\draw (1) edge["\tiny $m_3$"] (3);
\draw (4) edge["\tiny $m_2$"] (1);
\draw (0.1,1.9) arc (-75:255:.35);
\draw (-3.8,-1.3) arc (75:375:.35);
\draw (3.45,-1.3) arc (105:-195:.35);
\filldraw[white] (6,0) circle (0pt);
\filldraw[white] (-6,0) circle (0pt);
    \end{tikzpicture}

        \begin{tikzpicture}
\draw (0,0)++(90:1.58) node(2){$(4s^{-1},s^{-1},-3s^{-1})$};
\draw (0,0) node(1){$(s^{-1},s^{-1},-3s^{-1})$};
\draw (0,0)++(210:2.5) node(3){$(s^{-1},s^{-1},0)$};
\draw (0,0)++(330:2.5) node(4){$(s^{-1},4s^{-1},-3s^{-1})$};
\filldraw[black] (0,2.6) circle (0pt) node[anchor=south]{\tiny $m_2,m_3$};
\filldraw[black] (4.25,-1.7) circle (0pt) node[anchor=west]{\tiny $m_1,m_3$};
\filldraw[black] (-3.8,-1.7) circle (0pt) node[anchor=east]{\tiny $m_1,m_2$};
\draw (2) edge["\tiny $m_1$"] (1);
\draw (1) edge["\tiny $m_3$"] (3);
\draw (4) edge["\tiny $m_2$"] (1);
\draw (0.125,1.9) arc (-75:255:.35);
\draw (-3.3,-1.3) arc (75:375:.35);
\draw (3.8,-1.3) arc (105:-195:.35);
\filldraw[white] (6,0) circle (0pt);
\filldraw[white] (-6,0) circle (0pt);
    \end{tikzpicture}
\caption{
Orbits of size 4 for $a=(2,2,-2)$ and $s\neq 0$.
}
    \label{fig:size4}
\end{figure}

As for orbits of size 4, there are four ``tripods" of the same shape; one central triple connected to three double fixed points. In fact, these four tripods can be broken into two classes, one class of a single orbit where all three of the double fixed points contain one coordinate equal to 0 and another class containing three orbits where only one double fixed point contains a 0 coordinate. These can be seen in Figure~\ref{fig:size4}.

One thing to note is that all of these examples of tiny orbits of size 1, 2, and 4 are defined even when investigating Equation~(\ref{eqn:family}) over characteristic 0. Their existence is not governed by the specific choice of prime but rather by the global behaviour of solutions over $\overline{\mathbb{Q}}$. Table~\ref{table:22-2} contains the sizes of all non-trivial orbits for $a=(2,2,-2)$ and $p \leq 43$. However for the case of $p=5$, it follows that $s=5$ so the tiny orbits of size 1, 2, and 4 described above do not occur. For a similar breakdown of examples of small orbits over $\overline{\mathbb{Q}}$ in a family of Markoff-type K3 surfaces, see \cite{FLST}.

\begin{table}[h!]
    \centering
    \begin{tabular}{c|l}
        $p$ & Sizes of Orbits \\ \hline 
        2 & 4 \\
        3 & $1^3$, $2^3$ \\
        5 & $12^2$ \\
        7 & $1^3$, $2^3$, $4^3$, $8^3$ \\
        11 & $1^3$, $2^3$, $4^3$, $12^4$, $16^3$ \\
        13 & $1^3$, $2^3$, $4^4$, $16^3$, $24^4$  \\
        17 & $1^3$, $2^3$, $4^4$, $8^3$, $32^3$, $36^4$ \\
    \end{tabular}
    \qquad \qquad
    \begin{tabular}{c|l}
        $p$ & Sizes of Orbits \\  \hline 
        19 & $1^3$, $2^3$, $4^3$, $12^4$, $36^4$, $48^3$ \\
        23 & $1^3$, $2^3$, $4^3$,$8^3$, $16^3$, $60^4$, $64^3$ \\
        29 & $1^3$, $2^3$, $4^3$, $12^4$, $24^4$, $96^7$ \\
        31 & $1^3$, $2^3$, $4^3$, $8^3$, $12^4$, $32^3$, $96^4$, $128^3$\\
        37 & $1^3$, $2^3$, $4^3$, $16^3$, $36^4$, $144^3$, $180^4$ \\
        41 & $1^3$, $2^3$, $4^3$, $8^3$, $12^4$, $24^4$, $48^3$, $192^7$  \\
        43 & $1^3$, $2^3$, $4^3$, $24^4$, $60^4$, $192^4$, $240^3$ \\
    \end{tabular}
    
    \caption{For each prime $p$, we calculate the orbits for (\ref{eqn:family}) with $a_1=2$, $a_2=2$, $a_3=-2$. Entry $c^d$ indicates there are $d$ components of size $c$.}
    \label{table:22-2}
\end{table}

\subsection{Example: $(0,0,-3)$}
An especially interesting example is to take
\[
a_1=a_2=0, a_3=-3.
\]
Modulo 3, this choice would make all parameters $a_1=a_2=a_3=0$ and the moves would simply be sign changes $x_i \mapsto -x_i$.
The resulting graph is a cube as drawn in Figure~\ref{fig:mod3}.
The number of non-zero triples is 8, not a multiple of 3.
Although there is just one orbit for $p=3$, there can be many for larger primes.

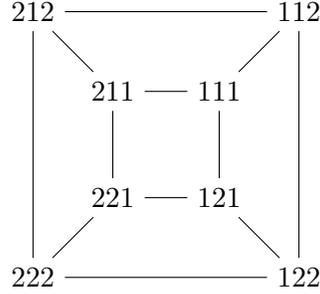
\begin{figure}
\begin{tikzpicture}
\pgfmathsetmacro{\r}{2.5}
\draw (0,0)++(45:1) node(111){111};
\draw (0,0)++(135:1) node(211){211};
\draw (0,0)++(-45:1) node(121){121};
\draw (0,0)++(-135:1) node(221){221};
\draw (0,0)++(45:\r) node(112){112};
\draw (0,0)++(135:\r) node(212){212};
\draw (0,0)++(-45:\r) node(122){122};
\draw (0,0)++(-135:\r) node(222){222};
\draw (111)--(211)--(221)--(121)--(122)--(112)--(212)--(222)--(122);
\draw (211)--(212);
\draw (111)--(112);
\draw (221)--(222);
\draw (111)--(121);
\end{tikzpicture}
\caption{
The solutions to $x^2+y^2+z^2=0 \bmod 3$.
}
\label{fig:mod3}
\end{figure}

The move on $i=3$ is a sign change
\[
m_3: x_3 \mapsto -x_3
\]
which commutes with the other two moves.
They are given by
\begin{align}
m_1: x_1 &\mapsto -x_1 + 3x_2 \\
m_2: x_2 &\mapsto -x_2 + 3x_1
\end{align}
Scaling $x \mapsto tx$ also commutes with all three moves.
It is therefore enough to understand the action on the two conics $(*,*,1)$ and $(*,*,-1)$, which are linked by $m_3$, and separately the action on $(*,*,0)$.
Overall, there may be several orbits for $(*,*,0)$, several orbits for $(*,*,\pm 1)$, and $\frac{p-1}{2}$ copies of the latter for $(*,*,\pm t)$ with $t \neq 0$.

In matrix form, the moves act by
\[
m_1 = \begin{pmatrix} -1 & 3 \\ 0 & 1 \end{pmatrix}, \quad m_2 = \begin{pmatrix} 1 & 0 \\ 3 & -1 \end{pmatrix}
\]
acting on $\begin{pmatrix} x_1 \\ x_2 \end{pmatrix}$.
Their product is
\[
m_1m_2 = \begin{pmatrix} 8 & -3 \\ 3 & -1 \end{pmatrix}, \quad m_2m_1 = (m_1m_2)^{-1} = \begin{pmatrix} -1 & 3 \\ -3 & 8 \end{pmatrix}
\]
with characteristic equation
\[
\lambda^2 - 7\lambda + 1 = 0.
\]
Typically there are two eigenvalues
\begin{equation}
\lambda = \frac{7+3\sqrt{5}}{2}
\end{equation}
for two choices of $\sqrt{5}$, either in $\F_p$ or a quadratic extension. By considering $\theta = \frac{3+\sqrt{5}}{2}$ which lives in the same field as $\lambda$, we have
    \begin{align*}
        \theta^2 = 3\theta -1 = \lambda
    \end{align*}
which implies the order of $\lambda$ is a divisor of $\dfrac{p\pm1}{2}$, not just $p\pm1$.

\begin{proposition} \label{prop:00-3} Let $p>5$ be prime.
\begin{enumerate}
    \item The number of orbits for $m_1$ and $m_2$ acting on the conic $(*,*,1)$ is
\[
\frac{1}{2} \frac{p\pm 1}{\ord(\lambda)} + \begin{cases} 1 & \mathrm{if}\ \sqrt{5} \in \F_p \\
    0 & \mathrm{if}\ \sqrt{5} \not\in \F_p \end{cases}\ ;
\]
    \item The number of orbits for $m_1$ and $m_2$ acting on the conic $(*,*,0)$ is
\[
\begin{cases}
    \dfrac{p - 1}{\ord(\lambda)} & \mathrm{if}\ \sqrt{5} \in \F_p \\
    0 & \mathrm{if}\ \sqrt{5} \not\in \F_p
\end{cases}
\]

\end{enumerate}
where $\lambda = \frac{7+3\sqrt{5}}{2}$, with multiplicative order $\ord(\lambda)$ in $\F_{p^2}^{\times}$ dividing $\frac{p-1}2$ if $\lambda \in \F_p^{\times}$ or $\frac{p+1}2$ if $\lambda$ lies in a quadratic extension.
\end{proposition}

It follows from quadratic reciprocity that there is a $\sqrt{5}$ in $\F_p$ if and only if $p \equiv 1$ or $4 \bmod 5$.

\begin{proof}
Recall the lemma of Burnside--Cauchy--Frobenius: the number of orbits for an action of a finite group is equal to the average number of fixed points \cite{Burnside,Cauchy,Frobenius}.
We apply this to the group generated by $m_1$ and $m_2$ acting on the conics $(*,*,1)$ and $(*,*,0)$.

The moves $m_1$ and $m_2$ generate a dihedral group of order $2\ord(\lambda)$, where $m_1m_2$ acts as a rotation and $m_2$ acts as a reflection.
In both cases the identity fixes all elements of the conic, where in the $(*,*,1)$ situation there are $p\pm1$ points and in the $(*,*,0)$ situation there are $2(p-1)$ or 0 points depending on whether or not $\sqrt{5} \in \F_p$. In both scenarios, the non-trivial rotations fix nothing. On the conic $(*,*,1)$, each reflection has either $2$ or $0$ fixed points, depending on whether the quadratic $5x_2^2-9$ has roots in $\F_p$.

Therefore the total number of fixed points for $(*,*,1)$ is
\[
p\pm 1 + \ord(\lambda) \cdot 0 + \ord(\lambda) \cdot \begin{cases} 2 & \mathrm{if}\ \sqrt{5} \in \F_p \\
    0 & \mathrm{if}\ \sqrt{5} \not\in \F_p \end{cases}  
\]
and dividing by $2\ord(\lambda)$ gives the number of orbits as claimed.

On the conic $(*,*,0)$, each reflection has 0 fixed points. This is because each triple is already fixed by $m_3$. By Lemma~\ref{lemma:N=1 iff a_i^2=4}, the existence of a double fixed point on a triple with third coordinate equal to 0 would imply $a_3^2=4$, which is not true if $p>5$ since $a_3^2=9$.
The result follows since the total number of fixed points is $0$ if $\sqrt{5} \not \in \F_p$ or $2(p-1)$ otherwise.
\end{proof}

Proposition~\ref{prop:00-3} implies that the conic $(*,*,0)$, if non-empty, always breaks into at least two pieces. As a concrete example where the $(*,*,\pm1)$ pieces decompose further, consider $p=89$. Taking $z=\pm 1$ gives the conic
\[
x^2+y^2-3xy +1 = 0
\]
which is a hyperbola with $p-1=88$ points rather than $p+1$.
It turns out that $\ord(\lambda)=11$.
There are 5 orbits in total, 3 of size 22 plus 2 of size 11. Figure~\ref{fig:p89-orbits} provides a graphical representation of some of these orbits.

\begin{figure}[h]
\centering
\begin{tikzpicture}[scale=1, every node/.style={font=\scriptsize}]
\tikzset{
    point/.style={circle, fill=black, inner sep=1.1pt},
    medge/.style={line width=0.35pt},
    zedge/.style={line width=0.3pt, densely dotted}
}
\begin{scope}
    \node at (0,1.25) {$22$ points with $z=1$};
    \node at (0,-1.25) {$22$ points with $z=-1$};
    \foreach \i in {0,...,21} {
        \pgfmathsetmacro{\ang}{-360*\i/22}
        \coordinate (a\i) at ({1.75*cos(\ang)},{0.48*sin(\ang)+0.45});
        \coordinate (b\i) at ({1.75*cos(\ang)},{0.48*sin(\ang)-0.45});
        \node[point] at (a\i) {};
        \node[point] at (b\i) {};
        \draw[zedge] (a\i)--(b\i);
    }
    \foreach \i in {0,...,20} {
        \pgfmathtruncatemacro{\j}{\i+1}
        \draw[medge] (a\i)--(a\j);
        \draw[medge] (b\i)--(b\j);
    }
    \draw[medge] (a21)--(a0);
    \draw[medge] (b21)--(b0);
    \node[right] at (a0) {$(1,1)$};
    \node[left] at (a11) {$(-1,-1)$};
\end{scope}
\begin{scope}[xshift=4.5cm]
    \node at (1.65,1.25) {$11$ points with $z=1$};
    \node at (1.65,-1.25) {$11$ points with $z=-1$};
    \foreach \i in {0,...,10} {
        \coordinate (c\i) at ({0.33*\i},0.45);
        \coordinate (d\i) at ({0.33*\i},-0.45);
        \node[point] at (c\i) {};
        \node[point] at (d\i) {};
        \draw[zedge] (c\i)--(d\i);
    }
    \foreach \i in {0,...,9} {
        \pgfmathtruncatemacro{\j}{\i+1}
        \draw[medge] (c\i)--(c\j);
        \draw[medge] (d\i)--(d\j);
    }
    \draw[medge] (c0) .. controls +(-0.3,0.42) and +(0.3,0.42) .. (c0);
    \draw[medge] (d0) .. controls +(-0.3,-0.42) and +(0.3,-0.42) .. (d0);
    \draw[medge] (c10) .. controls +(-0.3,0.42) and +(0.3,0.42) .. (c10);
    \draw[medge] (d10) .. controls +(-0.3,-0.42) and +(0.3,-0.42) .. (d10);
    \node[above] at (c5) {$(-31,31)$};
    \node[right] at (c10) {$(28,42)$};
\end{scope}
\end{tikzpicture}
\caption{
Orbits for $m_1$, $m_2$, $m_3$ acting on the conics $(*,*,\pm 1)$ for $p=89$ and $a=(0,0,-3)$.
The dotted vertical edges are the sign change $m_3$.
The horizontal edges are generated by $m_1$ and $m_2$ on each conic.
}
\label{fig:p89-orbits}
\end{figure}

One orbit of size $22$ contains both $(1,1)$ and $(-1,-1)$. This orbit is preserved by the involutions
\begin{align*}
(x,y, 1) \mapsto (y,x, 1) \\
(x,y,1) \mapsto (-x,-y,1) \\
\end{align*}
However, $(x,y)=(6,29)$ and $(x,y)=(-6,-29)$ lie in different orbits of size 22 which are images of each other under $(x,y) \mapsto (-y,-x)$.
The points $(-31,31)$ and $(31,-31)$ lie in two orbits of size 11, bookended by fixed points such as $(28,42)$ or $(-42,-9)$. On these points either $m_1$ or $m_2$ acts by
\[
\pm 42 \mapsto \pm 42.
\]


\begin{thebibliography}{99}

 \bibitem{BG} Esther Banaian and Yashaki Gyoda, \emph{Cluster algebraic interpretation of generalized Markov numbers and their matrixizations}, arXiv 2507.06900 \url{https://arxiv.org/abs/2507.06900}

 \bibitem{BV} Esther Banaian and Yadira Valdivieso, \emph{Snake Graphs and Caldero-Chapoton Functions from Triangulated Orbifolds}, Journal of Algebra,
 Volume 674, 15 July 2025, Pages 77-116
 \url{https://doi.org/10.1016/j.jalgebra.2025.03.006}

\bibitem{benny-goldman} Robert L. Benedetto and William M. Goldman, \emph{The topology of the relative character varieties of a quadruply-punctured sphere} Experiment. Math. 8(1): 85-103 (1999)

\bibitem{BGS} Jean Bourgain, Alexander Gamburd, and Peter Sarnak, 
\emph{Strong approximation and Diophantine properties of Markoff triples},
J. Amer. Math. Soc. 
    DOI: \url{https://doi.org/10.1090/jams/1061 }
    Published electronically: August 29, 2025  arXiv:1607.01530

\bibitem{Burnside} William Burnside (1897). Theory of Groups of Finite Order. Cambridge University Press

\bibitem{Carlitz} Leonard Carlitz, \emph{The number of points on certain cubic surfaces over a finite field}, Boll. Un. Mat. Ital., {\bf (3)}, 12, (1957), pages 19--21

\bibitem{Cauchy} Augustin-Louis Cauchy, \emph{M\'emoire sur diverses propri\'et\'es remarquables des substitutions r\'eguli\`eres ou irr\'eguli\`eres, et des syst\'emes de substitutiones conjug\'ees.} Comptes Rendus Acad. Sci. Paris 21, 835, 1845

\bibitem{Cantat} Serge Cantat, \emph{Bers and Hénon, Painlevé and Schrödinger}. Duke Math. J. 149 (3) 411 - 460, 15 September 2009. \url{https://doi.org/10.1215/00127094-2009-042}

\bibitem{Cantat-Loray} Serge Cantat and Frank Lorary, \emph{
Dynamics on Character Varieties and Malgrange irreducibility of Painlevé
VI equation
}
Annales de l'Institut Fourier, Volume 59 (2009) no. 7, pp. 2927-2978.
DOI \url{10.5802/aif.2512}

\bibitem{CDMB} Serge Cantat, Christophe Dupont, and Flrestan Martin-Baillon, \emph{Dynamics on Markov surfaces: classification of stationary measures}, arXiv pre-prints \url{https://arxiv.org/abs/2404.01721}

\bibitem{CMR} Leonid O. Chekhov, Marta Mazzocco, and Vladimir N. Rubtsov, \emph{Painlev\'e Monodromy Manifolds, Decorated Character Varieties, and Cluster Algebras}, Int. Math. Res. Not. 2017, No. 24, 7639-7691 (2017). \url{https://doi.org/10.1093/imrn/rnw219}

\bibitem{Chen}  William Y. Chen, \emph{Nonabelian level structures, Nielsen equivalence, and Markoff triples}, Ann. of Math. {\bf 199}, pages 301--443 (2024)
\url{https://doi.org/10.4007/annals.2024.199.1.5}

\bibitem{EFLMT} Jillian Eddy, Elena Fuchs, Matthew Litman, Daniel Martin, and Nico Tripeny, \emph{Connectivity of Markoff mod-p graphs and maximal divisors}, Proc. London Math. Soc. Volume 130, Issue 2, February 2025, e70027 \url{https://doi.org/10.1112/plms.70027}

\bibitem{Frobenius} Ferdinand Georg Frobenius (1887), "Ueber die Congruenz nach einem aus zwei endlichen Gruppen gebildeten Doppelmodul", Crelle's Journal, 101 (4): 273–299, doi:10.3931/e-rara-18804.

\bibitem{FLST} Elena Fuchs, Matthew Litman,  Joseph H. Silverman, and Austin Tran \emph{Orbits on K3 Surfaces of Markoff Type}. Experimental Mathematics (2023). DOI: 10.1080/ 10586458.2023.2239265.

\bibitem{G} Yasuaki Gyoda, \emph{Positive integer solutions to} $(x+y)^2+(y+z)^2+(z+x)^2=12xyz$, arXiv 2109.09639 \url{https://arxiv.org/abs/2109.09639}

\bibitem{GM} Yashaki Gyoda and Kodai Matsushita, \emph{Generalization of Markov Diophantine equation via generalized cluster algebra}, Electronic Journal of Combinatorics 30 (2023), P4.10

\bibitem{GMS} Yasuaki Gyoda, Shuhei Maruyama, and Yusuke Sato. \emph{SL(2,Z)-matrixizations of generalized Markov numbers} arXiv:2407.08203 \url{https://arxiv.org/abs/2407.08203}

\bibitem{Hikami} Kazuhiro Hikami, \emph{Note on Character Varieties and Cluster Algebras}, SIGMA {\bf 15} (2019), 003, 32 pages
\url{https://doi.org/10.3842/SIGMA.2019.003},

\bibitem{LLRS} Kyungyong Lee, Li Li, Michelle Rabideau, and Ralf Schiffler. \emph{On the ordering of the Markov numbers}, Advances in Applied Mathematics,
Volume 143, (2023), 102453.

\bibitem{Magnus} Wilhelm Magnus, \emph{Rings of Fricke Characters and Automorphism Groups
of Free Groups} Math. Z. 170, 91-103 (1980) \url{https://doi.org/10.1007/BF01214715}

\bibitem{MPT} Sara Maloni, Fr\'ed\'eric Palesi, and Ser Peow Tan. \emph{On the character variety of the four-holed sphere}, Groups Geom. Dyn. 9 (2015), 737–782
DOI 10.4171/GGD/326

\bibitem{Markoff} Andrey Markoff, \emph{Sur les formes quadratiques binaires ind\'efinies}, Volume 17, pages 379–399, (1880)  \url{https://link.springer.com/article/10.1007/BF01446234}

\bibitem{Martin} Daniel E. Martin, \emph{A new proof of Chen's theorem for Markoff graphs}, Inventiones Mathematicae, Volume 241, pages 623–626 (2025). \url{https://doi.org/10.1007/s00222-025-01346-9}

\bibitem{O} Evan M O`Dorney, \emph{Large Orbits on Markoff-Type K3 Surfaces over Finite Fields}, International Mathematics Research Notices, Volume 2023, Issue 24, (2023), pages 21874-–21879, \url{https://doi.org/10.1093/imrn/rnac341}

\bibitem{RS} Michelle Rabideau and Ralf Schiffler, \emph{Continued fractions and orderings on the Markov numbers}, Adv. Math. 370
(2020), 107231.
\url{https://doi.org/10.1016/j.aim.2020.107231}

\bibitem{Swinn} Peter Swinnerton-Dyer, \emph{Cubic surfaces over finite fields} Math. Proc. Camb. Phil. Soc. (2010), 149, 385
DOI \url{https://doi.org/10.1017/S0305004110000320}

\bibitem{Vogt} Henry Gustave Vogt, \emph{Sur les invariants fondamentaux des \'equations diff\'erentielles lin\'eaires du second ordre}.
Ann. Sci. Ecole Norm. Sup. (3) 6, Suppl. 3-72 (1889) (Th\`ese, Paris).

\end{thebibliography}
\end{document}